\documentclass[12pt,reqno]{amsart}
\usepackage[left=3cm,top=2cm,right=3cm,bottom=2cm]{geometry}
\usepackage{amssymb}
\usepackage{amsbsy}
\usepackage{epsfig}
\usepackage{tikz}

\usepackage[english]{babel}    

\usepackage{multicol}
\usepackage{blkarray}
\usepackage{mathtools}
\usepackage{amsthm}
\newtheorem{teore}{Theorem}[section]
\newtheorem{defn}[teore]{Definition}
\newtheorem{lemat}[teore]{Lemma}
\newtheorem{coro}[teore]{Corollary}

\newtheorem{prop}[teore]{Proposition}
\begin{document}
\DeclarePairedDelimiter\ceil{\lceil}{\rceil}
\DeclarePairedDelimiter\floor{\lfloor}{\rfloor}
\title{Independence and Matching Numbers of Unicyclic Graphs from Null Space}
\author[L. E. Allem]{L. Emilio Allem}\email{emilio.allem@ufrgs.br}
\address{UFRGS - Universidade Federal do Rio Grande do Sul, Instituto de Matem\'atica, Porto Alegre, Brazil}
\author[D. A. Jaume]{Daniel A. Jaume}\email{djaume@unsl.edu.ar}
\address{Universidad Nacional de San Luis, Departamento de Matem\'aticas, San Luis, Argentina}
\author[G. Molina]{Gonzalo Molina}\email{lgmolina@unsl.edu.ar}
\address{Universidad Nacional de San Luis, Departamento de Matem\'aticas, San Luis, Argentina}
\author[M. M. Toledo]{Maikon M. Toledo}\email{maikon.toledo@ufrgs.br}
\address{UFRGS - Universidade Federal do Rio Grande do Sul, Instituto de Matem\'atica, Porto Alegre, Brazil}
\author[V. Trevisan]{Vilmar Trevisan}\email{trevisan@mat.ufrgs.br }
 \address{UFRGS - Universidade Federal do Rio Grande do Sul, Instituto de Matem\'atica, Porto Alegre, Brazil}

\maketitle

\begin{abstract}
We characterize unicyclic graphs that are singular using the support of the
null space of their pendant trees.  From this, we obtain closed formulas
for the independence and matching numbers of a unicyclic graph, based on
the support of its subtrees. These formulas allows one to compute
independence and matching numbers of unicyclic graphs using linear algebra
methods.

\end{abstract}

\section{Introduction}

Recently, in \cite{tree}, the authors studied the null space of the adjacency matrix of trees and they presented a null decomposition of trees. In general, this null decomposition divides a tree into two forests (one of the forests can be empty), one composed by singular trees and the other composed by non-singular trees.
The technique used was the analysis of the \emph{support} of the tree, where the \emph{support} is defined as the subset of vertices for which at least one of its corresponding coordinates of the eigenvectors of the null space of the adjacency matrix is nonzero.

As an application, in \cite{tree}, the null decomposition was used to obtain closed formulas for two classical parameters.  The first one is the independence number of a graph $G$, denoted by $\alpha(G)$. Notice that the problem of computing $\alpha(G)$ is $NP$-hard \cite{10007212609} and several mathematicians have studied $\alpha(G)$ ( for example \cite{Alon:1998:AIN:286563.286564,Frieze:1990:INR:82922.82935,SHEARER198383}). The second one is the matching number of a graph $G$, denoted by $\nu(G)$ \cite{doob,ming}. Historically the matching theory started with bipartite graphs and one of the earliest works was published in 1916 \cite{konig}.

In this paper, we extend the results of \cite{tree} to unicyclic graphs. In a more general sense,  we obtain structural information of the unicyclic graphs using the support of their subtrees. In particular, we obtain closed formulas for the independence and matching numbers of unicyclic graphs that depend on the support and the core of their subtrees. Next we give an outline of this paper. It is worth pointing out that, in practice, this means that these classical parameters can be computed using linear algebra.

In section \ref{sec:bas}, we present some basic notations and definitions of support of a graph. In section \ref{sec:sing}, we characterize singular unicyclic graphs using the support of their pendant trees.
In section \ref{sec:alpha}, we obtain a closed formula for the independence number of unicyclic graphs using the support of the subtrees of these unicyclic graphs.
In section \ref{sec:nu}, we obtain closed formulas for the matching number of unicyclic graphs based on the  support of subtrees of these unicyclic graphs.

\section{Basic definitions and notation}\label{sec:bas}

In this section we present some notation and basic definitions. In particular, we explain the notion of support of a graph. We use the graphs of Figure \ref{f5} to illustrate the concepts used here.

Let $G=(V,E)$ be a simple graph of order $n$, with vertex set $V=\lbrace{v_1,\ldots,v_n}\rbrace$ and edge set $E=E(G)$, the adjacency matrix $A(G)=(a_{ij})_{n\times{n}}$ of $G$ is defined as
$$a_{ij}=\begin{cases}
1, & \mbox{if }\lbrace{v_i,v_j}\rbrace\in{E};\\
0, & \mbox{if }\lbrace{v_i,v_j}\rbrace\notin{E}.
\end{cases}$$

Denote by $\varepsilon_{\lambda}$ the $\lambda$-eigenspace of $A(G)$, thus $\varepsilon_{\lambda}=\lbrace{x\in{\mathbb{R}^n}: A(G)x=\lambda{x}}\rbrace$. The  $0$-eigenspace  ($\varepsilon_0$) is the focus of our work and will be denote by $\mathcal{N}(G)$.
The \emph{nullity of a graph $G$}, denoted by $\eta(G)$, is the multiplicity of the eigenvalue zero in the spectrum of $A(G)$ or, equivalently, the dimension of the $0$-eigenspace of $G$. The graph $G$ is called singular if $A(G)$ is a singular matrix or $\eta(G)>0.$ Otherwise, the graph $G$ is called non-singular.

As an example, we observe that the set
\noindent $S=\lbrace{(0,1,0,-1,0,0)^{t},(0,0,1,-1,0,0)^{t}}\rbrace$ is a basis for the null space of the tree $T_1$ of Figure \ref{f5}, hence $\eta(T_1) = 2$. And we notice that 0 is not an eigenvalue of the tree $T_2$ of Figure \ref{f5}, so $\eta(T_2)= 0$.

\begin{defn}
A set $I\subset{V} $ of vertices of a graph $G$ is an independent set in $G$ if no two vertices in $I$ are adjacent. A maximum independent set is an independent set of maximum cardinality.
The cardinality of any maximum independent set in $G$, denoted by $\alpha(G)$, is called the \emph{independence number of $G$}.
$\mathcal{I}(G)$ denotes the set of all maximum independent sets of $G$.
\end{defn}

For example, in Figure \ref{f5} the vertex subsets
$\lbrace{v_2,v_3,v_4,v_5}\rbrace$ and $\lbrace{v_2,v_3,v_4,v_6}\rbrace$  of
the tree $T_1$ are the only independent sets of maximum cardinality.
Therefore,
$$\mathcal{I}(G)=\lbrace{\lbrace{v_2,v_3,v_4,v_5}\rbrace,\lbrace{v_2,v_3,v_4,v_6}
\rbrace}\rbrace,$$ and
$\alpha(T_1)=\vert\lbrace{v_2,v_3,v_4,v_5}\rbrace\vert=4$.

\begin{defn}
A matching $M$ in $G$ is a set of pairwise non-adjacent edges, that is, no
two edges in $M$ share a common vertex. A maximum matching is a matching of
largest cardinality in $G$. The matching number of $G$, denoted by $\nu(G)$,
is the size of a set of any maximum matching. $\mathcal{M}(G)$ denotes the
set of all maximum matching of $G$. A vertex is saturated by $M$, if it is an
endpoint of one of the edges in the matching $M$. Otherwise the vertex is
said non-saturated. Moreover, a matching is said to be perfect if it
saturates all vertices of $G$.
\end{defn}

In the figures, we use zig zag edges to represent the edges of a matching.

\noindent In Figure \ref{f5}, the tree $T_1$ has matching
$\lbrace{\lbrace{v_1,v_3}\rbrace,\lbrace{v_5,v_6}\rbrace}\rbrace$ and the
tree $T_2$ has perfect matching
$\lbrace{\lbrace{v_7,v_9}\rbrace,\lbrace{v_8,v_{10}}\rbrace,\lbrace{v_{11},v_{14}}
\rbrace,\lbrace{v_{12},v_{13}}\rbrace}\rbrace$. Therefore, $\nu(T_1)=2$ and
$\nu(T_2)=4$.

The Edmond-Gallai vertices of $G$, denoted by $EG(G)$, is the set of all
vertices of $G$ that are non-saturated by some maximum matching $M$ in $G$.

\noindent For example, the maximum matchings of $T_1$ are $$\mathcal{M}(T_1)=\lbrace{\lbrace{\lbrace{v_1,v_2}\rbrace,\lbrace{v_5,v_6}\rbrace}\rbrace, \lbrace{\lbrace{v_1,v_3}\rbrace,\lbrace{v_5,v_6}\rbrace}\rbrace,\lbrace{\lbrace{v_1,v_4}\rbrace,\lbrace{v_5,v_6}\rbrace}\rbrace}\rbrace.$$ Thus $EG(T_1)=\lbrace{v_2,v_3,v_4}\rbrace$ and $EG(T_2)=\emptyset.$
\usetikzlibrary{shapes,snakes}
\tikzstyle{vertex}=[circle,draw,minimum size=5pt,inner sep=10pt]
\tikzstyle{edge} = [draw,thick,-]
\tikzstyle{matched edge} = [draw,snake=zigzag,line width=1pt,-]
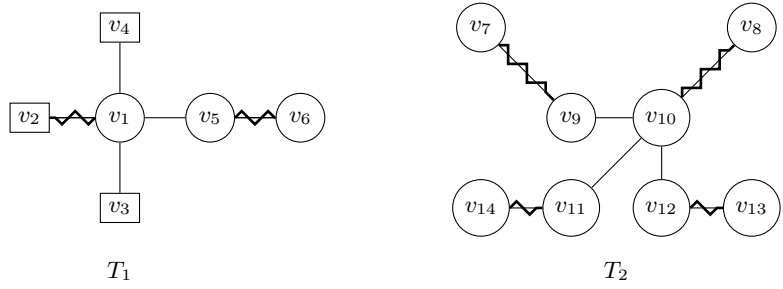
\begin{figure}[h!]
\begin{scriptsize}
\begin{center}
\begin{tikzpicture}[scale=1.2,auto,swap]
\node[draw,circle,label=below left:] (1) at (0,0) {$v_{1}$};
\node[draw,rectangle,label=below left:] (2) at (-1,0) {$v_{2}$};
\node[draw,rectangle,label=below left:] (3) at (0,-1) {$v_3$};
\node[draw,rectangle,label=below left:] (4) at (0,1) {$v_{4}$};
\node[draw,circle,label=below left:] (5) at (1,0) {$v_{5}$};
\node[draw,circle,label=below left:] (6) at (2,0) {$v_{6}$};
\node[draw,circle,label=below left:] (7) at (4,1) {$v_{7}$};
\node[draw,circle,label=below left:] (9) at (5,0) {$v_{9}$};
\node[draw,circle,label=below left:] (10) at (6,0) {$v_{10}$};
\node[draw,circle,label=below left:] (8) at (7,1) {$v_{8}$};
\node[draw,circle,label=below left:] (14) at (4,-1) {$v_{14}$};
\node[draw,circle,label=below left:] (11) at (5,-1) {$v_{11}$};
\node[draw,circle,label=below left:] (12) at (6,-1) {$v_{12}$};
\node[draw,circle,label=below left:] (13) at (7,-1) {$v_{13}$};
\node at (0,-1.7) {$T_1$};
\node at (5.5,-1.7) {$T_2$};
\foreach \from/\to in {1/2,1/3,1/4,1/5,5/6,9/10,10/12,11/10,7/9,10/8,12/13,14/11,1/2,5/6}{
 \draw (\from) -- (\to);}
 \foreach \source / \dest in {7/9,10/8,12/13,14/11,1/2,5/6}
   \path[matched edge] (\source) -- (\dest);
\end{tikzpicture}
\end{center}
\end{scriptsize}
 \caption{Support, matching and independent set.}
\label{f5}
\end{figure}

\begin{defn} Let $G$ be a graph with $n$ vertices and let $x$ be a vector of $\mathbb{R}^{n}$. The support of $x$ in $G$ is
$$Supp_{G}(x)=\lbrace{v\in{V(G)}:x_{v}\neq{0}}\rbrace.$$

\noindent Let $S$ be a subset of $\mathbb{R}^{n}$. Then the support of $S$ in
 $G$ is $$Supp_{G}(S)=\bigcup_{x\in{S}}{Supp_{G}(x)}.$$\end{defn}

As a convention, we use rectangular vertices in figures to represent
vertices of the support. Consider the tree $T_1$ (Figure \ref{f5}) and the
set of vectors

$S=\lbrace{(0,1,0,-1,0,0)^{t},(0,0,1,-1,0,0)^{t}}\rbrace.$ Then
$Supp_{T_1}(S)=\lbrace{v_2,v_3,v_4}\rbrace.$

The following result shows that in order to compute the support of an eigenspace of $A(G)$, it is enough to analyse the coordinates of the vectors of a basis of this eigenspace.
\begin{lemat}\cite{tree}\label{suppcompute}
Let $G$ be a graph, and $\lambda$ an eigenvalue of $A(G)$. Let $\mathcal{B}=\lbrace{b_1,\ldots,b_k}\rbrace$ be a basis of $\varepsilon_{\lambda}$, then $Supp_G(\varepsilon_{\lambda})=Supp_G(\mathcal{B}).$
\end{lemat}

We are interested in the support of the null space of $A(G)$,
that is, our focus is $Supp_G(\mathcal{N}(G))$, which, for purposes of
notation, is denoted by $Supp(G)$. In practice to compute $Supp(G)$ we will
use Lemma \ref{suppcompute} and not the definition of support. That is, we
compute a basis of the null space and consider the entries of the vectors in
the basis to obtain the support.

For example, notice that $S=\lbrace{(0,1,0,-1,0,0)^{t},(0,0,1,-1,0,0)^{t}}\rbrace$  is a basis of $\mathcal{N}(T_1)$ (Figure \ref{f5}), thus $$Supp(T_1)=Supp_{T_1}(S)=\lbrace{v_2,v_3,v_4}\rbrace.$$
Moreover, note that $T_2$ (Figure \ref{f5}) is non-singular, that is, $\mathcal{N}(T_2)=\lbrace{0}\rbrace$, therefore, $Supp(T_2)=\emptyset$.

\begin{teore}\cite{tree}\label{tindependent}
Let $T$ be a tree, then $Supp(T)$ is an independent set of $T$.
\end{teore}

The following well known result characterizes singular trees in terms of
their matchings.

\begin{lemat}\label{c2}\cite{singulartree}
$T$ is a nonsingular tree if and only if $T$ has a perfect matching.
\end{lemat}

As we can see in Figure \ref{f5}, the tree $T_1$ is a singular tree, because
it does not have perfect matching. The tree $T_2$ is a non-singular tree,
because it has perfect matching.

Our first goal is to characterize singular unicyclic graphs in terms of the
support of their pendant trees, which is the subject of next section.

Using Lemma \ref{c2} above, and Lemma 3.12, Corollary 4.14 of \cite{tree} we obtain
the next result.

\begin{lemat}\label{t11}
Let $T$ be a tree, then $EG(T)=Supp(T)$.
\end{lemat}

It means that only the vertices of the support of a tree are not saturated by
some maximum matching in this tree.

\section{Singular Unicyclic Graphs}\label{sec:sing}

In this section we characterize singular unicyclic graphs using the support
of their pendant trees, which is the statement of Theorem \ref{c4}.

For cycles, the problem of characterizing singular graphs is solved.

\begin{lemat}\label{t1}\cite{ciclo}
A cycle $C_n$ of $n$ vertices is singular if and only if $n$ is divisible by $4$.
\end{lemat}

Hence, for the remaining of this section, we will consider $G$ a unicyclic
graph $G\neq{C_n} $. Let $G$ be a unicyclic graph and let $C$ be the unique
cycle of $G$. For each vertex $v\in{V(C)}$, we denote by $G\lbrace{v}\rbrace$
the induced connected subgraph of $G$ with maximum possible number of
vertices, which contains the vertex $v$ and no other vertex of $C$.
$G\lbrace{v}\rbrace$ is called the \emph{pendant tree of $G$ at $v$}. Notice
that $G$ is obtained by identifying the vertex $v$ of $G\lbrace{v}\rbrace$
with the vertex $v$ on $C$ for all vertices $v\in{C}$. In Figure \ref{f8} we
have two unicyclic graphs $G$ and $H$ with their pendant trees
$G\lbrace{v_1}\rbrace$, $G\lbrace{v_2}\rbrace$, $G\lbrace{v_3}\rbrace$,
$H\lbrace{v_4}\rbrace$, $H\lbrace{v_5}\rbrace$ and $H\lbrace{v_6}\rbrace$,
respectively.

\tikzstyle{vertex}=[circle,draw,minimum size=2pt,inner sep=2pt]
\tikzstyle{edge} = [draw,thick,-]
\tikzstyle{matched edge} = [draw,line width=11pt,-]
\tikzstyle{same edge} = [draw,-]
\begin{figure}[h!]
\begin{scriptsize}
\begin{center}
\begin{tikzpicture}[scale=0.7,auto,swap]
\draw[dashed] (0.6,0.3) -- (3.4,1.5);
\draw[dashed] (0.2,0.65) --(0.7,0.65) (0.7,0.65) --(2.2,3.4);
\draw[dashed] (0.2,0.65) --(-1.3,3.4);
\draw[dashed] (-1.3,3.4) --(2.2,3.4);
\draw[dashed] (0.6,0.3) -- (0.6,-0.5);
\draw[dashed] (0.6,-0.5) --(3.4,-1.5);
\draw[dashed] (3.4,-1.5)--(3.4,1.5);
\draw[dashed] (0.35,0.3) -- (0.35,-0.5);
\draw[dashed] (0.35,-0.5)--(-3.7,-0.5);
\draw[dashed] (-3.7,-0.5) --(-2.3,2.3);
\draw[dashed] (0.35,0.3) --(-2.3,2.3);
\node at (4.1,0) {$G\lbrace{v_2}\rbrace$};
\node at (-4,0.8) {$G\lbrace{v_1}\rbrace$};
\node at (0.5,4) {$G\lbrace{v_3}\rbrace$};
\node at (0.5,-2) {$G$};
    \foreach \pos/\name in {{(0,0)/v_1}, {(1,0)/v_2}, {(0.5,1)/v_3},{(-1,0)/{i}},{(-2,0)/{j}},{(-3,0)/{l}},{(2,0)/{m}},{(3,0)/{n}},{(-1,0.8)/q},{(-2,0.8)/r},{(-2,1.5)/a}, {(3,-0.8)/b},{(3,0.6)/{c}}, {(0.5,2)/d}, {(0.5,3)/f},{(-0.6,3)/g},{(1.6,3)/h}}{\node[vertex] (\name) at \pos {$\name$};}
    \foreach \source/ \dest in {v_1/v_2, v_1/v_3, v_2/v_3, v_1/i, i/j, l/j, m/n, v_2/m, m/c, v_2/b, v_1/q, q/r, q/a, v_3/d,d/f,d/g,d/h}
    \path[edge] (\source) -- (\dest);
\draw[dashed] (9.2,0.65) --(9.7,0.65) (9.7,0.65) --(11.2,3.4);
\draw[dashed] (9.2,0.65) --(7.7,3.4);
\draw[dashed] (7.7,3.4) --(11.2,3.4);
\draw[dashed] (9.6,0.4) -- (12.4,0.4);
\draw[dashed] (9.6,0.3) -- (9.6,-0.4);
\draw[dashed] (9.6,-0.4) -- (12.4,-0.4);
\draw[dashed] (12.4,0.3) -- (12.4,-0.4);
\draw[dashed] (9.35,0.3) -- (9.35,-0.5);
\draw[dashed] (9.35,-0.5)--(6.7,-0.3);
\draw[dashed] (6.7,-0.3) --(6.7,2.3);
\draw[dashed] (9.35,0.3) --(6.7,2.3);
\node at (13.2,0) {$H\lbrace{v_5}\rbrace$};
\node at (6,0.8) {$H\lbrace{v_4}\rbrace$};
\node at (9.5,4) {$H\lbrace{v_6}\rbrace$};
\node at (9.5,-2) {$H$};
    \foreach \pos/\name in {{(9,0)/v_4}, {(10,0)/v_5}, {(9.5,1)/v_6},{(8,0)/{z}},{(7,0)/{t}},{(11,0)/{x}},{(12,0)/{y}},{(8,0.8)/w},{(7,0.8)/s},{(7,1.5)/u}, {(9.5,2)/o},{(8.4,3)/e},{(10.6,3)/p}}{\node[vertex] (\name) at \pos {$\name$};}
    \foreach \source/ \dest in {v_4/v_5, v_4/v_6, v_5/v_6, v_4/z, z/t, x/y, v_5/x, v_4/w, w/s, w/u, v_6/o,o/e,o/p}
    \path[edge] (\source) -- (\dest);
\end{tikzpicture}
\end{center}
\end{scriptsize}
 \caption{Unicyclic graphs of Type $I$ and $II$ and their pendant trees.}
\label{f8}
\end{figure}
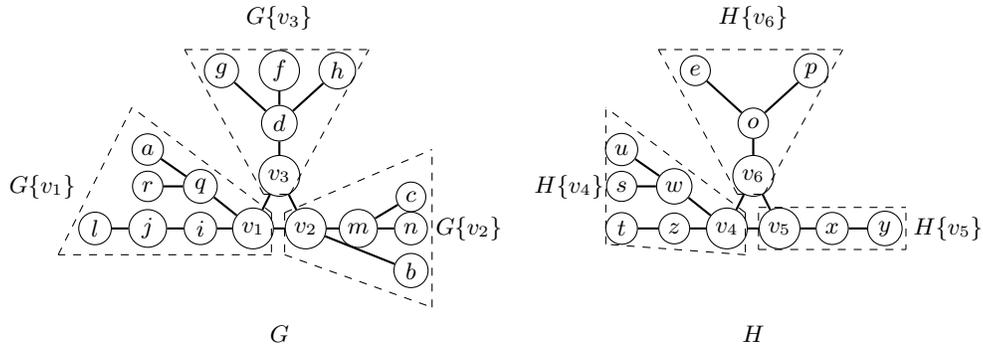

\begin{defn}\cite{nulidade}
For a tree $G\lbrace{v}\rbrace$ with at least two vertices, the vertex
$v\in{G\lbrace{v}\rbrace}$ is called mismatched in $G\lbrace{v}\rbrace$ if
there exists a maximum matching of $G\lbrace{v}\rbrace$ that does not
saturate $v$; otherwise, $v$ is called matched in $G\lbrace{v}\rbrace$. If a
tree consists of only one vertex it is considered mismatched.

A unicyclic graph $G$ is said be of Type $I$ if there exists a vertex $v$ on
the cycle of $G$ such that $v$ is matched in $G\lbrace{v}\rbrace$, otherwise,
$G$ is said to be of Type $II$.
\end{defn}

To emphasize, a unicyclic graph $G$ is of Type $I$, if there exists a vertex
$v$ of its cycle that is saturated by all maximum matchings of the pendant
tree $G\lbrace{v}\rbrace$. $G$ is of Type $II$ if any vertex $v$ of its cycle
is not saturated by some maximum matching of $G\lbrace{v}\rbrace$.

As an example, consider the unicyclic graph $G$ in Figure \ref{f8}. We notice
that $G$ is of Type $I$, because the vertex $v_1$ is matched in
$G\lbrace{v_1}\rbrace$. Indeed, the maximum matchings of
$G\lbrace{v_1}\rbrace$ are
$\lbrace{\lbrace{a,q}\rbrace,\lbrace{v_1,i}\rbrace,
\lbrace{j,l}\rbrace}\rbrace$ and
$\lbrace{\lbrace{r,q}\rbrace,\lbrace{v_1,i}\rbrace,\lbrace{j,l}\rbrace}
\rbrace$ and both saturate $v_1$. On the other hand, the unicyclic graph $H$
of Figure \ref{f8} is of Type $II$, because the pendant trees
$H\lbrace{v_4}\rbrace$, $H\lbrace{v_5}\rbrace$ and $H\lbrace{v_6}\rbrace$
have maximum matchings that do not saturate $v_4$, $v_5$ and $v_6$,
respectively. For example,
$\lbrace{\lbrace{u,w}\rbrace,\lbrace{t,z}\rbrace}\rbrace$,
$\lbrace{\lbrace{x,y}\rbrace}\rbrace$ and
$\lbrace{\lbrace{p,o}\rbrace}\rbrace$ are maximum matchings in
$H\lbrace{v_4}\rbrace$, $H\lbrace{v_5}\rbrace$ and $H\lbrace{v_6}\rbrace$,
respectively, that do not saturate $v_4$, $v_5$ and $v_6$, respectively.

We show next that in order to verify that a unicyclic graph is Type $I$ or
$II$, it suffices to check whether a vertex $v$ of the cycle is or is not in
the support of the pendant tree $G\lbrace{v}\rbrace$.

\begin{prop}\label{t4}
A unicyclic graph $G$ is of Type $I$ if and only if there exists at least one
pendant tree $G\lbrace{v}\rbrace$ such that
$v\notin{Supp(G\lbrace{v}\rbrace)}$.
\end{prop}

\begin{proof}
Since $G$ is of Type $I$ we know that there exists a vertex $v$ in the cycle
of $G$ such that $v$ is always saturated by any maximum matching in
$G\lbrace{v}\rbrace$, that is, $v\notin{EG(G\{v\})}$, by Lemma \ref{t11} we
have $v\notin{Supp{(G\lbrace{v}\rbrace)}}$.
\end{proof}

Immediately, we obtain the dual result.

\begin{coro}\label{c3}
A unicyclic  graph $G$ is of Type $ II $ if and only if every pendant tree
$G\lbrace{v}\rbrace$ is such that $v\in{Supp(G\lbrace{v}\rbrace)}$.
\end{coro}

The following result computes the nullity of a unicyclic graph from the
nullity of its pendant trees.

\begin{lemat}\label{t6}\cite{nulidade}
Let $G$ be a unicyclic graph and let $C$ be its cycle. If $G$ is of Type $I$ and $v\in{V(C)}$ be matched in $G\lbrace{v}\rbrace$, then $$\eta(G)=\eta(G\lbrace{v}\rbrace)+\eta(G-G\lbrace{v}\rbrace).$$
If $G$ is of Type $II$ then $$\eta(G)=\eta(G-C)+\eta(C).$$
\end{lemat}

We now obtain a characterization of singular unicyclic graphs using the
support of their pendant trees.

\begin{teore}\label{c4}
Let $G$ be a unicyclic graph and let $C$ be the cycle of $G$. $G$ is singular
if and only if one of the following happens:
\begin{itemize}
\item[\textbf{(i)}] There is a pendant tree $G\lbrace{v}\rbrace$, with
    $v\notin{Supp(G\lbrace{v}\rbrace)}$ and either $G\lbrace{v}\rbrace$
    does not have perfect matching or $G-G\lbrace{v}\rbrace$ does not have
    perfect matching;

\item[\textbf{(ii)}] Every pendant tree $G\lbrace{v}\rbrace$, is such that
    $v\in{Supp(G\lbrace{v}\rbrace)}$ and either one of the trees that
    compose the forest $G-C$ does not have perfect matching or the cycle
    $C$ has length equal to a multiple of $4$.
    \end{itemize}
\end{teore}
\begin{proof}

\begin{itemize}
\item[(i)] As there is a pendant tree $G\lbrace{v}\rbrace$ such that
    $v\notin{Supp(G\lbrace{v}\rbrace)}$, we conclude, by Proposition
    \ref{t4} that $G$ is of Type $I$. Moreover, by Lemma \ref{t6} we
    conclude that $G$ will be singular if and only if
    $G\lbrace{v}\rbrace$ or $G-G\lbrace{v}\rbrace$ has
    nonzero nullity, and by Lemma \ref{c2} we know that this only happens
    if $G\lbrace{v}\rbrace$ does not have perfect matching or
    $G-G\lbrace{v}\rbrace$ does not have perfect matching;
\item[(ii)] Since every pendant tree $G\lbrace{v}\rbrace$, is such that
    $v\in{Supp(G\lbrace{v}\rbrace)}$ we conclude by Corollary \ref{c3} that
    $G$ is of Type $II$. Then, by Lemma \ref{t6} we know that $G$ will be singular 
		if and only if $G-C$ or $C$ has nonzero nullity, and by the Lemma \ref{c2} and Lemma \ref{t1} it happens if
    and only if at least one of the trees composing the forest $G-C$ does
    not have perfect matching or the cycle $C$ has length equal to a multiple of
    $4$.
\end{itemize}
\end{proof}

\section{Independence number of Unicyclic graphs}\label{sec:alpha}

In this section we obtain closed formulas for the independence number of a
unicyclic graph $G$. This formula depends on the number of vertices of the
support and also on the amount of $N$-vertices of subtrees of this unicyclic
graph $G$. In order to understand our result, we start by presenting the null
decomposition of trees, given in \cite{tree}.

\begin{defn}
Let $T$ be a tree. The $S$-forest of $T$, denoted by $\mathcal{F}_S(T)$, is
defined as the subgraph induced by the closed neighborhood of $Supp(T)$ in
$T$:
$$\mathcal{F}_S(T)=T\langle{N[Supp(T)]}\rangle.$$
The $N$-forest  of $T$, denoted by $\mathcal{F}_N(T)$, is defined as the
remaining graph:
$$\mathcal{F}_N(T)=T-\mathcal{F}_S(T).$$ The Null Decomposition of $T$ is the
pair $\left(\mathcal{F}_S(T),\mathcal{F}_N(T)\right)$.
\end{defn}

\noindent $V(\mathcal{F}_N(T))$ is called the set of $N$-vertices of $T$.

We represent star vertices in the figures as the $N$-vertices. As an example,
the support of the tree $T$ in Figure \ref{figura2} is
$$Supp(T)=\lbrace{v_2,v_3,v_6,v_7,v_8,v_{10},v_{11},v_{12},v_{19},v_{21},v_{22}}\rbrace.$$
The $S$-forest of $T$ generated by the closed neighborhood of the support
consists of
\begin{eqnarray*}
\mathcal{F}_S(T)&=&T\langle{N[Supp(T)]}\rangle=T\langle{N[\lbrace{v_2,v_3,v_6,v_7,v_8,v_{10},v_{11},v_{12},v_{19},v_{21},v_{22}}\rbrace]}\rangle\\
&=&T\langle{\lbrace{v_1,v_2,v_3,v_4,v_5,v_6,v_7,v_8,v_9,v_{10},v_{11},v_{12},v_{19},v_{20},v_{21},v_{22}}\rbrace}\rangle\\
&=&S_1\cup{S_2}\cup{S_3}\cup{S_4}.
\end{eqnarray*}
The $N$-forest of $T$ consists of:
\begin{eqnarray*}
\mathcal{F}_N(T)&=&T-\mathcal{F}_S(T)\\
&=&T\langle{\lbrace{v_{13},v_{14},v_{15},v_{16},v_{17},v_{18}}\rbrace}\rangle\\
&=&N_1\cup{N_2}.
\end{eqnarray*}

Figure \ref{figura2} illustrates the null decomposition of the tree $T$.

\usetikzlibrary{shapes,snakes}
\tikzstyle{vertex}=[circle,draw,minimum size=1pt,inner sep=1pt]
\tikzstyle{edge} = [draw,thick,-]
\tikzstyle{matched edge} = [draw,snake=zigzag,line width=1pt,-]
\begin{figure}[h!]
\begin{scriptsize}
\begin{tikzpicture}[scale=1.3,auto,swap]
\node[draw,rectangle,label=below left:] (2) at (0,0) {$v_2$};
\node[draw,circle,label=below left:] (1) at (1,0) {$v_1$};
\node[draw,rectangle,label=below left:] (3) at (1,1) {$v_3$};
\node[draw,star,star points=9,star point ratio=0.6,label=below left:] (13) at (2,0) {$v_{13}$};
\node[draw,star,star points=9,star point ratio=0.6,label=below left:] (14) at (3,0) {$v_{14}$};
\node[draw,circle,label=below left:] (4) at (4.5,0) {$v_4$};
\node[draw,rectangle,label=below left:] (6) at (4,1) {$v_6$};
\node[draw,rectangle,label=below left:] (7) at (5,1) {$v_7$};
\node[draw,circle,label=below left:] (5) at (6.5,0) {$v_5$};
\node[draw,rectangle,label=below left:] (8) at (7,1) {$v_8$};
\node[draw,circle,label=below left:] (9) at (2,-1) {$v_9$};
\node[draw,star,star points=9,star point ratio=0.6,label=below left:] (16) at (3,-1) {$v_{16}$};
\node[draw,star,star points=9,star point ratio=0.6,label=below left:] (15) at (4,-1) {$v_{15}$};
\node[draw,star,star points=9,star point ratio=0.6,label=below left:] (17) at (3,-2) {$v_{17}$};
\node[draw,star,star points=9,star point ratio=0.6,label=below left:] (18) at (3,-3) {$v_{18}$};
\node[draw,rectangle,label=below left:] (10) at (1,-1) {$v_{10}$};
\node[draw,rectangle,label=below left:] (11) at (1,-2) {$v_{11}$};
\node[draw,rectangle,label=below left:] (12) at (2,-2) {$v_{12}$};
\node[draw,rectangle,label=below left:] (19) at (6,1) {$v_{19}$};
\node[draw,circle,label=below left:] (20) at (5,-2) {$v_{20}$};
\node[draw,rectangle,label=below left:] (21) at (5,-1) {$v_{21}$};
\node[draw,rectangle,label=below left:] (22) at (5,-3) {$v_{22}$};
\draw[dashed] (-0.3,0.3) -- (0.7,0.3);
\draw[dashed] (-0.3,0.3) -- (-0.3,-0.3);
\draw[dashed] (-0.3,-0.3) -- (1.3,-0.3);
\draw[dashed] (1.3,-0.3) -- (1.3,1.3);
\draw[dashed] (0.7,1.3) -- (1.3,1.3);
\draw[dashed] (0.7,1.3) -- (0.7,0.3);
\node at (0.3,0.7) {$S_1$};
\draw[dashed] (1.7,0.3) -- (3.3,0.3);
\draw[dashed] (1.7,0.3) -- (1.7,-0.3);
\draw[dashed] (1.7,-0.3) -- (3.3,-0.3);
\draw[dashed] (3.3,0.3) -- (3.3,-0.3);
\node at (2.5,0.5) {$N_1$};
\draw[dashed] (0.7,-0.7) -- (0.7,-2.3);
\draw[dashed] (0.7,-0.7) -- (2.3,-0.7);
\draw[dashed] (2.3,-0.7) -- (2.3,-2.3);
\draw[dashed] (0.7,-2.3) -- (2.3,-2.3);
\node at (0.5,-1.5) {$S_3$};
\draw[dashed] (2.7,-3.3) -- (2.7,-0.7);
\draw[dashed] (2.7,-0.7) -- (4.3,-0.7);
\draw[dashed] (4.3,-0.7) -- (4.3,-1.3);
\draw[dashed] (4.3,-1.3) -- (3.3,-1.3);
\draw[dashed] (3.3,-1.3) -- (3.3,-3.3);
\draw[dashed] (3.3,-3.3) -- (2.7,-3.3);
\node at (3,-3.5) {$N_2$};
\draw[dashed] (4.2,-0.3) -- (6.8,-0.3);
\draw[dashed] (3.6,1.3) -- (7.4,1.3);
\draw[dashed] (6.8,-0.3) -- (7.4,1.3);
\draw[dashed] (4.2,-0.3) --(3.6,1.3);
\node at (5.5,1.5) {$S_2$};
\draw[dashed] (4.7,-0.7) --(4.7,-3.3);
\draw[dashed] (5.3,-0.7) --(5.3,-3.3);
\draw[dashed] (4.7,-0.7) --(5.3,-0.7);
\draw[dashed] (4.7,-3.3) --(5.3,-3.3);
\node at (5.5,-2) {$S_4$};
\node at (3,-4) {$T$};
\foreach \from/\to in {1/2,1/3,1/13,13/14,13/9,9/10,9/11,9/12,13/14,14/4,4/6,4/7,5/8,9/16,16/15,16/17,17/18,1/13,4/14,9/13,9/16,4/14,4/5,5/19,20/21,20/22,17/20} {
 \draw (\from) -- (\to);}
 \foreach \source / \dest in {1/2,6/4,8/5,9/10,21/20,17/18,13/14,15/16}
   \path[matched edge] (\source) -- (\dest);
\end{tikzpicture}
 \begin{center}
\caption{Null decomposition of the tree $T$.}\label{figura2}
\end{center}
\end{scriptsize}
\end{figure}
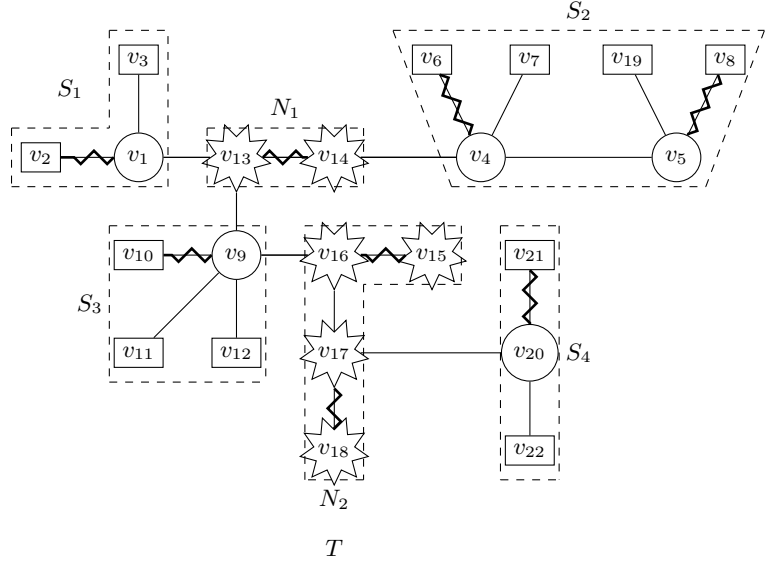

\begin{defn}
The core of $G$, denoted by $Core(G)$, is defined to be the set of all the neighbours of the supported vertices of $G$:
$$Core(G)=\bigcup_{v\in{Supp(G)}}{N(v)}.$$
\end{defn}

\noindent For example, the core of tree $T$ (Figure \ref{figura2}) is:
$$Core(T)=\lbrace{v_1,v_4,v_5,v_9,v_{20}}\rbrace.$$

The next lemma gives closed formulas for the independence and matching
numbers of trees and it is crucial to prove our main  results.

\begin{lemat}\label{t35}\cite{tree} Let $T$ be a tree. Then
\begin{eqnarray*}
\nu(T)&=&\vert{Core(T)}\vert+\frac{\vert{V(\mathcal{F}_N(T))}\vert}{2}\\
\alpha(T)&=&\vert{Supp(T)}\vert+\frac{\vert{V(\mathcal{F}_N(T))}\vert}{2}.
\end{eqnarray*}
\end{lemat}

Our next result tells us that given any vertex in a non-singular tree, there
will always be at least one maximum independent set that does not contain
this vertex and another maximum independent set that contains this vertex.

\begin{prop}\label{ttt30}
Let $T$ be a non-singular tree and $v\in{V(T)}$. Then there exist
$I_1,I_2\in\mathcal{I}(T)$ such that $v\in{I_1}$ and
$v\notin{I_2}$.
\end{prop}

\begin{proof}
Since $T$ is a tree we have that $T$ is a bipartite graph. Then there are two disjoint subsets $B_1$ and $B_2$ of $V(T)$ such that $V(T)=B_1\cup{B_2}$ and for all $\lbrace{a,b}\rbrace\in{E(T)}$ we have $\lbrace{a,b}\rbrace\cap{B_1}\neq\emptyset$ and $\lbrace{a,b}\rbrace\cap{B_2}\neq\emptyset$. As $T$ is a non-singular tree, it has perfect matching $M$.
As $\alpha(T)=\nu(T)=\vert{M}\vert=\frac{\vert{V(T)}\vert}{2}$ and for all $\lbrace{a,b}\rbrace\in{M}$ we have $\lbrace{a,b}\rbrace\cap{B_1}\neq\emptyset$ and $\lbrace{a,b}\rbrace\cap{B_2}\neq\emptyset$, then $\vert{B_1}\vert=\vert{B_2}\vert=\frac{\vert{V(T)}\vert}{2}$. That is, $B_1,B_2\in\mathcal{I}(T)$. Therefore, given a $v\in{V(T)}$ we have $v\in{B_1}$ and $v\notin{B_2}$ or $v\in{B_2}$ and $v\notin{B_1}$.
\end{proof}

\begin{lemat}\label{t30}
If $T$ is a tree and $v\in{V(\mathcal{F}_N(T))}$ then there exist
$I_1,I_2\in\mathcal{I}(T)$ such that $v\in{I_1}$ and
$v\notin{I_2}$.
\end{lemat}

\begin{proof}
The null decomposition, in general, divides a tree $T$ into two
forests (one them may be empty), a forest formed by singular trees, denoted
by  $\mathcal{F}_S(T)$, and other formed by non-singular trees, denoted by
$\mathcal{F}_N(T)$ (see Theorem 4.5 and Theorem 4.13 of \cite{tree}).

Moreover, we have $Supp(T)=\bigcup\limits_{S\in\mathcal{F}_S(T)}Supp(S)$ and
$V(\mathcal{F}_N(T))=\bigcup\limits_{N\in\mathcal{F}_N(T)}V(N)$ . Thus, if
$v\in{V(\mathcal{F}_N(T))}$ we have $v\in{V(N_1)}$ for some
$N_1\in\mathcal{F}_N(T)$. As $N_1$ is non-singular using Proposition
\ref{ttt30} we obtain $I_{N_1},J_{N_1}\in\mathcal{I}(N_1)$ such that
$v\in{J_{N_1}}$ and $v\notin{I_{N_1}}$.

Let \begin{eqnarray*}
I_1&=&Supp(T)\cup{I_{N_1}}\cup\left(\bigcup\limits_{N\in\mathcal{F}_N(T)\mbox{ and }N\neq{N_1}}I_N\right)\mbox{ and}\\ I_2&=&Supp(T)\cup{J_{N_1}}\cup\left(\bigcup\limits_{N\in\mathcal{F}_N(T)\mbox{ and }N\neq{N_1}}J_N\right).
\end{eqnarray*}
We observe that $I_1$ and $I_2$ are independent sets, because $Supp(T)$ is an
independent set of $T$ and $N(Supp(T))=Core(T)$, then
$N(Supp(T))\cap{\left(\bigcup\limits_{N\in\mathcal{F}_N(T)}I_N\right)}=\emptyset$.
Notice that
$\vert{I_1}\vert=\vert{I_2}\vert=\vert{Supp(T)}\vert+\frac{\vert{V(\mathcal{F}_N(T))}\vert}{2}$,
then  $I_1,I_2\in\mathcal{I}(T)$ by Lemma \ref{t35}. Moreover, we have
$v\notin{I_1}$ and $v\in{I_2}$.
\end{proof}

\begin{lemat}\label{tcore}
Let $T$ be a tree and $I$ an independent set of $T$. If $c_i\in{Core(T)}$ and
$c_i\in{I}$ then $I\notin\mathcal{I}(T)$.
\end{lemat}

\begin{proof}
We notice that, in general,
$I=\lbrace{s_1,s_2,\ldots,s_j}\rbrace\cup\lbrace{c_1,c_2,\ldots,c_t}\rbrace\cup\lbrace{n_1,n_2,\ldots,n_r}\rbrace$,
where $\lbrace{s_1,s_2,\ldots,s_j}\rbrace\subseteq{Supp(T)}$,
$\lbrace{c_1,c_2,\ldots,c_t}\rbrace\subseteq{Core(T)}$ and
$\lbrace{n_1,n_2,\ldots,n_r}\rbrace\subseteq{V(\mathcal{F}(T))}$ (possibly we
can have $\lbrace{n_1,n_2,\ldots,n_r}\rbrace=\emptyset$). By Lemma 3.5 of
\cite{tree} we have
$$\vert{N(\lbrace{c_1,c_2,\ldots,c_t}\rbrace)\cap{Supp(T)}}\vert>\vert{\lbrace{c_1,c_2,\ldots,c_t}\rbrace}\vert.$$
Note that  $J=\lbrace{s_1,s_2,\ldots,s_j}\rbrace\cup\left(N(\lbrace{c_1,c_2,\ldots,c_t}\rbrace)\cap{Supp(T)}\right)\cup\lbrace{n_1,n_2,\ldots,n_r}\rbrace$ is an independent set of $T$. Indeed, $N(Supp(T))\cap{\lbrace{n_1,n_2,\ldots,n_r}\rbrace}=\emptyset$  and $\lbrace{s_1,s_2,\ldots,s_j}\rbrace$ and $N(\lbrace{c_1,c_2,\ldots,c_t}\rbrace)\cap{Supp(T)}$ are independent sets, because  $\lbrace{s_1,s_2,\ldots,s_j}\rbrace\subseteq{Supp(T)}$, $N(\lbrace{c_1,c_2,\ldots,c_t}\rbrace)\cap{Supp(T)}\subseteq{Supp(T)}$ and $Supp(T)$ is an independent set by Theorem \ref{tindependent}. Moreover, note that $\lbrace{s_1,s_2,\ldots,s_j}\rbrace\cap\left(N(\lbrace{c_1,c_2,\ldots,c_t}\rbrace)\cap{Supp(T)}\right)=\emptyset$ otherwise $I$ would not be an independent set, then $\vert{J}\vert>{\vert{I}\vert}$, therefore, $I\notin\mathcal{I}(T)$.
\end{proof}

Theorem \ref{propalpha1} is one of the main results of this section. It gives
a closed formula for the independence number of unicyclic graphs $G$ of Type
$I$. This formula depends on the support and $N$-vertices of subtrees. It
means that using this formula we can compute the independence number
of unicyclic graphs of Type $I$ using linear algebra.

\begin{teore}\label{propalpha1}
If $G$ is a unicyclic graph of Type $I$ and $G\lbrace{v}\rbrace$ its pendant
tree such that $v\notin{Supp(G\lbrace{v}\rbrace)}$ then
$$\alpha{(G)}=\vert{Supp(G\lbrace{v}\rbrace)}\vert+\vert{Supp(G-G\lbrace{v}\rbrace)}\vert{+}\frac{\vert{V(\mathcal{F}_N(G\lbrace{v}\rbrace))}\vert{+}\vert{V(\mathcal{F}_N(G-G\lbrace{v}\rbrace))}\vert}{2}\\.$$
\end{teore}

\begin{proof}
Note that there is an independent set
$I_1\in\mathcal{I}(G\lbrace{v}\rbrace)$ such that $v\notin{I_1}$. Indeed, if
$v\in{V(\mathcal{F}_N(G\lbrace{v}\rbrace))}\cup{Core(G\lbrace{v}\rbrace)}$
then by lemmas \ref{t30} and \ref{tcore} there is a
$I_1\in\mathcal{I}(G\lbrace{v}\rbrace)$ such that $v\notin{I_1}$. Let
$I_2\in\mathcal{I}(G-G\lbrace{v}\rbrace)$. Let
$u,w\in{N(v)\cap{V(G-G\lbrace{v}\rbrace)}}$. We will prove that
$I_1\cup{I_2}\in\mathcal{I}(G)$.

We notice that $I_1\cup{I_2}$ is an independent set in $G$. To see that, we
observe that the vertices of $I_1$ are not connected to each other, because
$I_1$ is an independent set. Similarly, we conclude the same for $I_2$.
Moreover, the only adjacencies between $G\lbrace{v}\rbrace$ and
$G-G\lbrace{v}\rbrace$ occur between vertices $v$ and $u$ and vertices $v$
and $w$. Since $v\notin{I_1}$, there is no possibility of adjacency between
vertices of $I_1$ and vertices of $I_2$. Suppose that
$I_1\cup{I_2}\notin\mathcal{I}(G)$, that is, there exists an independent set
$J$ in $G$ such that $\vert{J}\vert>\vert{I_1\cup{I_2}}\vert$. As
$V(G)=V(G\lbrace{v}\rbrace)\cup{V(G-G\lbrace{v}\rbrace)}$ there is a
$J_1\subseteq{V(G\lbrace{v}\rbrace)}$  and
$J_2\subseteq{V(G-G\lbrace{v}\rbrace)}$ such that $J=J_1\cup{J_2}$. We have that
\begin{eqnarray}\label{equa4}
\vert{I_1}\vert+\vert{I_2}\vert=\vert{I_1\cup{I_2}}\vert<{\vert{J}\vert}=\vert{J_1}\vert+\vert{J_2}\vert.
\end{eqnarray}
In this case, we see that $J_1$ and $J_2$ are independent sets in
$G\lbrace{v}\rbrace$ and $G-G\lbrace{v}\rbrace$, respectively. Thus we have
$\vert{J_1}\vert\leq{\vert{I_1}\vert}$ and
$\vert{J_2}\vert\leq{\vert{I_2}\vert}$, because
$I_1\in\mathcal{I}(G\lbrace{v}\rbrace)$ and
$I_2\in\mathcal{I}(G-G\lbrace{v}\rbrace)$. Therefore,
$\vert{J_1}\vert+\vert{J_2}\vert\leq\vert{I_1}\vert+\vert{I_2}\vert$, which
is a contradiction by \eqref{equa4}. Hence, $I_1\cup{I_2}\in\mathcal{I}(G)$.
By Lemma \ref{t35}, we have
\begin{eqnarray*}
\alpha{(G\lbrace{v}\rbrace)}&=&\vert{Supp(G\lbrace{v}\rbrace)}\vert{+}\frac{\vert{V(\mathcal{F}_N(G\lbrace{v}\rbrace))}\vert}{2}\mbox{ and}\\
\alpha{(G-G\lbrace{v}\rbrace)}&=&\vert{Supp(G-G\lbrace{v}\rbrace)}\vert{+}\frac{\vert{V(\mathcal{F}_N(G-G\lbrace{v}\rbrace))}\vert}{2}.
\end{eqnarray*}
Therefore, the independence number of $G$ is given by
\begin{eqnarray*}
\alpha{(G)}&=&\vert{I_1}\vert+\vert{I_2}\vert
           =\alpha{(G\lbrace{v}\rbrace)}+\alpha{(G-G\lbrace{v}\rbrace)}\\
           &=&\vert{Supp(G\lbrace{v}\rbrace)}\vert{+}\frac{\vert{V(\mathcal{F}_N(G\lbrace{v}\rbrace))}\vert}{2}+\vert{Supp(G-G\lbrace{v}\rbrace)}\vert{+}\frac{\vert{V(\mathcal{F}_N(G-G\lbrace{v}\rbrace))}\vert}{2}.\\
\end{eqnarray*}
\end{proof}

The following example is an application of Theorem \ref{propalpha1}.  Consider $G$ the unicyclic graph of Figure \ref{figur3}. We observe that $G$ is of Type $I$. Indeed,
$v\notin{Supp(G\lbrace{v}\rbrace)=\lbrace{g,e,d}\rbrace}$, then by
Proposition \ref{t4} we have that $G$ is a unicyclic graph of Type $I$. Moreover,
$Supp(G-G\lbrace{v}\rbrace)=\lbrace{a,b,i,j,z}\rbrace$,
$V(\mathcal{F}_N(G\lbrace{v}\rbrace))=\emptyset$ and
${V(\mathcal{F}_N(G-G\lbrace{v}\rbrace)}=\lbrace{w,c}\rbrace$.

\tikzstyle{vertex}=[circle,draw,minimum size=2pt,inner sep=2pt]
\tikzstyle{edge} = [draw,thick,-] \tikzstyle{matched edge} = [draw,line
width=1pt,-,color=red]
\begin{figure}[h!]
\begin{center}
\begin{scriptsize}
\begin{center}
\begin{tikzpicture}[scale=1.0,auto,swap]
\node[draw,circle,label=below left:] (1) at (0,0.5) {$v$};
\node[draw,circle,label=below left:] (2) at (-2,-1) {$x$};
\node[draw,rectangle, label=below left:] (3) at (-1,-1) {$z$};
\node[draw,circle,label=below left:] (4) at (1,-1) {$u$};
\node[draw,star,star points=9,star point ratio=0.75,label=below left:] (5) at (2,-1) {$c$};
\node[draw,rectangle,label=below left:] (6) at (0.5,-2) {$a$};
\node[draw,rectangle,label=below left:] (7) at (1.5,-2) {$b$};
\node[draw,star,star points=9,star point ratio=0.75,label=below left:] (8) at (3,-1) {$w$};
\node[draw,rectangle,label=below left:] (9) at (1,0.5) {$d$};
\node[draw,rectangle,label=below left:] (10) at (0,1.3) {$e$};
\node[draw,rectangle,label=below left:] (12) at (-1,0.5) {$g$};
\node[draw,rectangle,label=below left:] (14) at (-2.8,-1.5) {$i$};
\node[draw,rectangle,label=below left:] (15) at (-2.8,-0.5) {$j$};
\draw[dashed] (-2,-0.7) --(3.3,-0.7);
\draw[dashed] (-2,-0.7) --(-2.6,-0.2);
\draw[dashed] (-2.6,-0.2)--(-3.1,-0.2);
\draw[dashed] (-3.1,-0.2)--(-3.1,-1.8);
\draw[dashed] (-3.1,-1.8)--(-2.6,-1.8);
\draw[dashed] (-2.6,-1.8)--(-2,-1.3);
\draw[dashed] (-2,-1.3)--(0.2,-1.3);
\draw[dashed] (0.2,-1.3)--(0.2,-2.3);
\draw[dashed] (0.2,-2.3)--(1.8,-2.3);
\draw[dashed] (1.8,-2.3)--(1.8,-1.3);
\draw[dashed] (1.8,-1.3)--(3.3,-1.3);
\draw[dashed] (3.3,-1.3)--(3.3,-0.7);
\draw[dashed] (-1.3,0.2)--(1.3,0.2);
\draw[dashed] (-1.3,0.2)--(-1.3,0.8);
\draw[dashed] (-1.3,0.8)--(-0.3,0.8);
\draw[dashed] (-0.3,0.8)--(-0.3,1.6);
\draw[dashed] (-0.3,1.6)--(0.3,1.6);
\draw[dashed] (0.3,1.6)--(0.3,0.8);
\draw[dashed] (0.3,0.8)--(1.3,0.8);
\draw[dashed] (1.3,0.8)--(1.3,0.2);
\node at (0,-2.6) {$G-G\lbrace{v}\rbrace$};
\node at (0,1.9) {$G\lbrace{v}\rbrace$};
\foreach \from/\to in {2/3,3/4,4/5,4/6,4/7,5/8,1/9,1/10,1/12,2/14,2/15,1/8,1/4} {
 \draw (\from) -- (\to);}
 \foreach \source / \dest in {}
   \path[matched edge] (\source) -- (\dest);
\end{tikzpicture}
\end{center}
\end{scriptsize}
 \end{center}
\begin{center}
\caption{Unicyclic graph $G$ and its subtrees.}\label{figur3}
\end{center}
\end{figure}
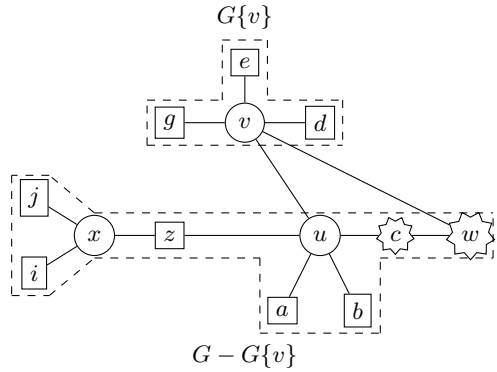

\noindent Therefore, by Theorem \ref{propalpha1}, we have that the independence number of $G$ is given by:
\begin{eqnarray*}
\alpha{(G)}&=&\vert{Supp(G\lbrace{v}\rbrace)}\vert+\vert{Supp(G-G\lbrace{v}\rbrace)}\vert{+}\frac{\vert{V(\mathcal{F}_N(G\lbrace{v}\rbrace))}\vert{+}\vert{V(\mathcal{F}_N(G-G\lbrace{v}\rbrace))}\vert}{2}\\
&=&3+5+\frac{2}{2}=9.
\end{eqnarray*}
We observe that $J=\lbrace{a,b,c,d,e,g,i,j,c}\rbrace$ is a maximum
independent set of $G$ and $\vert{J}\vert=9$.

\begin{lemat}\label{non supp}
Let $G$ be a unicyclic graph and $C$ its cycle. Let $G\lbrace{v}\rbrace$ be a
pendant tree such that $v\in{Supp(G\lbrace{v}\rbrace)}$. If
$u\in{N(v)\cap{V(G\lbrace{v}\rbrace)}}$ then $u\notin{Supp(G-C)}$.
\end{lemat}

\begin{proof}
Let $G\lbrace{v}\rbrace-v=\bigcup\limits_{i=1}^{k}T_i$, where $T_i$ is a
connected component of $G\lbrace{v}\rbrace-v$. Let $M_i\in\mathcal{M}(T_i)$.
As $v\in{Supp({G\lbrace{v}\rbrace})}$ there is a
$M\in\mathcal{M}(G\lbrace{v}\rbrace)$  such that $M$ does not saturate $v$ by
Lemma \ref{t11}. We observe that  $u\in{N(v)\cap{V(G\lbrace{v}\rbrace)}}$,
then $u\notin{Supp({G\lbrace{v}\rbrace})}$, because
$Supp({G\lbrace{v}\rbrace})$ is an independent set of $G\lbrace{v}\rbrace$ by
Theorem \ref{tindependent}. Then $M\in\mathcal{M}(G\lbrace{v}\rbrace-v)$ and
$\nu(G\lbrace{v}\rbrace)=\nu(G\lbrace{v}\rbrace-v)$. Suppose
$u\in{Supp(G\lbrace{v}\rbrace-v)=\bigcup\limits_{i=1}^{k}Supp(T_i)}$, that
is, there is a $i$ such that $u\in{Supp(T_i)}$, then by Lemma \ref{t11} in
$T_i$  we obtain a  $M_i\in\mathcal{M}(T_i)$ does not saturate $u$. Note
that, $\bigcup\limits_{i=1}^kM_i\in\mathcal{M}(G\lbrace{v}\rbrace-v)$  and
$\bigcup\limits_{i=1}^kM_i$ does not saturate $u$. Then
$\bigcup\limits_{i=1}^kM_i\in\mathcal{M}(G\lbrace{v}\rbrace)$, because
$\nu(G\lbrace{v}\rbrace)=\nu(G\lbrace{v}\rbrace-v)$. Which is a
contradiction, because  $\bigcup\limits_{i=1}^kM_i$ does not saturate $u$ and
all maximum matching in $G\lbrace{v}\rbrace$ saturates $u$ because
$u\notin{Supp(G\lbrace{v}\rbrace)}$. Therefore,
$u\notin{Supp(G\lbrace{v}\rbrace-v)}$.	Since the connected components of
$G\lbrace{v}\rbrace-v$ are connected components of $G-C$ we have
$u\notin{Supp(G-C)}$.
\end{proof}

Theorem \ref{propalpha2} is a similar result for unicyclic graphs of Type
$II$ and gives a closed formula for the independence number of unicyclic
graphs of Type $II$.

\begin{teore}\label{propalpha2}
Let $G$ be a unicyclic graph and $C$ its cycle. Let
$G-C=\bigcup\limits^{k}_{i=1} {T_i}$, where $T_i$ is a connected component of
$G-C$. If $G$ is a unicyclic graph of Type $II$ then
$$\alpha{(G)}=\floor*{\frac{\vert{V(C)}\vert}{2}}+\sum\limits^{k}_{i=1}\vert{Supp(T_i)}\vert{+}
\frac{\vert{V(\mathcal{F}_N(T_i))}\vert}{2}.$$
\end{teore}

\begin{proof}
Let $v_i\in{V(C)}$ and $u_i\in{T_i}$ such that $u_i\in{N(v_i)}$. By Lemma
\ref{non supp}, we have $u_i\notin{Supp(T_i)}$. As
$u_i\in{V(\mathcal{F}_N(T_i))}\cup{Core(T_i)}$ then, by lemmas \ref{t30} and
\ref{tcore}, we have a $J_i\in\mathcal{I}(T_i)$ such that $u_i\notin{J_i}$.
Consider $J_c\in\mathcal{I}(C)$ and define $I=J_c\cup\left(\bigcup\limits_{i=1}^{k}J_i\right)$. We will show that
$I\in\mathcal{I}(G)$.

First, we notice that $I$ is an independent set in $G$. Indeed, for all $i$ we
observe that the vertices of $J_i$ are not connected to each other, because
$J_i$ is an independent set. Similarly, we conclude the same for $J_c$.
Moreover, there is no chance that a vertex of $J_i$ is adjacent to a vertex
$J_{\ell}$, with $i\neq{\ell}$, since the vertices of the trees $T_i$ and
$T_{\ell} $ are not adjacent to each other. Now, we show that vertices of
$J_i$ are not connected to vertices of $J_c$. To see that, we observe that
the only adjacency that exists between $T_i$ and $C$ is the adjacency between
vertex $u_i$ and $v_i$, but since $u_i\notin{J_i}$, there is no possibility
of adjacency between vertices of $J_i$ and $J_c $.

Suppose now $I\notin\mathcal{I}(G)$, that is, there is an independent set $F$ in $G$
such that $\vert{F}\vert{>}{\vert{I}\vert}$. As
$V(G)=V(C)\cup\left(\bigcup\limits_{i=1}^{k}V(T_i)\right)$, we see that there
exist  $F_c\subseteq{V(C)}$ and $F_i\subseteq{V(T_i)}$ such that
$F=F_c\cup\left(\bigcup\limits_{i=1}^{k}F_i\right)$. Thus we have
\begin{eqnarray}\label{equa3}
\vert{F_c}\vert+\sum\limits_{i=1}^{k}\vert{F_i}\vert=\vert{F}\vert{>}\vert{I}\vert=\vert{J_c}\vert+\sum\limits_{i=1}^{k}\vert{J_i}\vert.
\end{eqnarray}
As $F_c$ and $F_i$ are independent sets of $C$ and $T_i$, respectively, we
have $\vert{F_c}\vert\leq{\vert{J_c}\vert}$ and
$\vert{F_i}\vert\leq{\vert{J_i}\vert}$, because $J_c\in\mathcal{I}(C)$ and
$J_i\in\mathcal{I}(T_i)$. Thus, we have
$\vert{F_c}\vert+\sum\limits_{i=1}^{k}\vert{F_i}
\vert\leq\vert{J_c}\vert+\sum\limits_{i=1}^{k}\vert{J_i}\vert$, which is a
contradiction by \eqref{equa3}. Therefore, $I\in\mathcal{I}(G)$. We observe
that $\alpha(C)=\floor*{\frac{\vert{V(C)}\vert}{2}}$ and, by Lemma \ref{t35},
we have, for all $i$,
$$\alpha{(T_i)}=\vert{Supp(T_i)}\vert{+} \frac{\vert{V(\mathcal{F}_N(T_i))}\vert}{2}.$$
Therefore, the independence number of $G$ is given by
\begin{eqnarray*}
\alpha{(G)}&=&\vert{J_c}\vert+\sum\limits_{i=1}^{k}\vert{J_i}\vert
           =\alpha{(C)}+\sum\limits^{k}_{i=1}\alpha(T_i)\\
           &=&\floor*{\frac{\vert{V(C)}\vert}{2}}+\sum\limits^{k}_{i=1}\vert{Supp(T_i)}\vert{+}\frac{\vert{V(\mathcal{F}_N(T_i))}\vert}{2}.\\
\end{eqnarray*}
\end{proof}

\tikzstyle{vertex}=[circle,draw,minimum size=2pt,inner sep=2pt]
\tikzstyle{edge} = [draw,thick,-]
\tikzstyle{matched edge} = [draw,line width=1pt,-,color=red]
\begin{figure}[h!]
\begin{center}
\begin{scriptsize}
\begin{center}
\begin{tikzpicture}[scale=1.0,auto,swap]
\node[draw,circle,label=below left:] (1) at (0,0) {$v$};
\node[draw,circle,label=below left:] (2) at (1,1) {$w$};
\node[draw,circle,label=below left:] (3) at (2,0) {$u$};
\node[draw,circle,label=below left:] (4) at (1,2) {$a$};
\node[draw,circle,label=below left:] (5) at (0,2) {$b$};
\node[draw,circle,label=below left:] (6) at (2,2) {$c$};
\node[draw,circle,label=below left:] (7) at (1,3) {$e$};
\node[draw,circle,label=below left:] (8) at (2,3) {$d$};
\node[draw,circle,label=below left:] (9) at (3,2) {$f$};
\node[draw,circle,label=below left:] (10) at (-1,2) {$g$};
\node[draw,circle,label=below left:] (11) at (0,3) {$h$};
\node[draw,circle,label=below left:] (12) at (-1,3) {$i$};
\node[draw,circle,label=below left:] (13) at (3,0) {$j$};
\node[draw,circle,label=below left:] (14) at (3,1) {$\ell$};
\node[draw,circle,label=below left:] (15) at (-1,1) {$m$};
\node[draw,circle,label=below left:] (16) at (-1,0) {$n$};
\node[draw,circle,label=below left:] (17) at (-2,1) {$o$};
\node[draw,circle,label=below left:] (18) at (-2,0) {$p$};
\node[draw,star,star points=9,star point ratio=0.6,label=below left:] (19) at (6,1) {$m$};
\node[draw,star,star points=9,star point ratio=0.6,label=below left:] (20) at (6,0) {$n$};
\node[draw,star,star points=9,star point ratio=0.6,label=below left:] (21) at (5,1) {$o$};
\node[draw,star,star points=9,star point ratio=0.6,label=below left:] (22) at (5,0) {$p$};
\node[draw,star,star points=9,star point ratio=0.6,label=below left:] (23) at (10,0) {$j$};
\node[draw,star,star points=9,star point ratio=0.6,label=below left:] (24) at (10,1) {$\ell$};
\node[draw,circle,label=below left:] (26) at (9,2) {$c$};
\node[draw,rectangle,label=below left:] (27) at (9,3) {$d$};
\node[draw,rectangle,label=below left:] (28) at (10,2) {$f$};
\node[draw,star,star points=9,star point ratio=0.6,label=below left:] (29) at (8,2) {$a$};
\node[draw,star,star points=9,star point ratio=0.6,label=below left:] (30) at (8,3) {$e$};
\node[draw,rectangle,label=below left:] (31) at (6,2) {$g$};
\node[draw,rectangle,label=below left:] (32) at (7,3) {$h$};
\node[draw,rectangle,label=below left:] (33) at (6,3) {$i$};
\node[draw,circle,label=below left:] (34) at (7,2) {$b$};
\draw[dashed] (1.7,0.3) --(2.6,0.3);
\draw[dashed] (2.6,0.3) --(2.6,1.4);
\draw[dashed] (2.6,1.4) --(3.4,1.4);
\draw[dashed] (3.4,1.4) --(3.4,-0.4);
\draw[dashed] (3.4,-0.4) --(1.7,-0.4);
\draw[dashed] (1.7,-0.4) --(1.7,0.3);
\draw[dashed] (-2.4,-0.4) --(-2.4,1.4);
\draw[dashed] (-2.4,-0.4) --(0.4,-0.4);
\draw[dashed] (0.4,-0.4) --(0.4,0.4);
\draw[dashed] (0.4,0.4) --(-0.9,1.4);
\draw[dashed] (-0.9,1.4) --(-2.4,1.4);
\draw[dashed] (-1.4,3.4) --(2.4,3.4);
\draw[dashed] (-1.4,3.4) --(-1.4,1.6);
\draw[dashed] (-1.4,1.6) --(0,1.6);
\draw[dashed] (0,1.6) --(0.8,0.6);
\draw[dashed] (0.8,0.6) --(1.3,0.6);
\draw[dashed] (1.3,0.6) --(2,1.6);
\draw[dashed] (2,1.6) --(3.4,1.6);
\draw[dashed] (3.4,1.6) --(3.4,2.4);
\draw[dashed] (3.4,2.4) --(2.4,2.4);
\draw[dashed] (2.4,2.4) --(2.4,3.4);
\draw[dashed] (5.6,3.4) --(7.4,3.4);
\draw[dashed] (7.4,3.4) --(7.4,1.6);
\draw[dashed] (7.4,1.6) --(5.6,1.6);
\draw[dashed] (5.6,1.6) --(5.6,3.4);
\draw[dashed] (7.6,3.4) --(7.6,1.6);
\draw[dashed] (7.6,1.6) --(8.4,1.6);
\draw[dashed] (8.4,1.6) --(8.4,3.4);
\draw[dashed] (8.4,3.4) --(7.6,3.4);
\draw[dashed] (8.6,3.4) --(8.6,1.6);
\draw[dashed] (8.6,1.6) --(10.4,1.6);
\draw[dashed] (10.4,1.6) --(10.4,2.4);
\draw[dashed] (10.4,2.4) --(9.4,2.4);
\draw[dashed] (9.4,2.4) --(9.4,3.4);
\draw[dashed] (9.4,3.4) --(8.6,3.4);
\draw[dashed] (9.6,1.4) --(10.4,1.4);
\draw[dashed] (10.4,1.4) --(10.4,-0.4);
\draw[dashed] (10.4,-0.4) --(9.6,-0.4);
\draw[dashed] (9.6,-0.4) --(9.6,1.4);
\draw[dashed] (4.6,1.4) --(6.4,1.4);
\draw[dashed] (6.4,1.4) --(6.4,-0.4);
\draw[dashed] (6.4,-0.4) --(4.6,-0.4);
\draw[dashed] (4.6,-0.4) --(4.6,1.4);
\node at (1,-1.3) {$G$};
\node at (8,-1.3) {$G-C_3$};
\node at (1,3.7) {$G\lbrace{w}\rbrace$};
\node at (-1,-0.6) {$G\lbrace{v}\rbrace$};
\node at (2.5,-0.6) {$G\lbrace{u}\rbrace$};
\node at (6.5,3.7) {$T_1$};
\node at (8,3.7) {$T_2$};
\node at (9,3.7) {$T_3$};
\node at (10,-0.6) {$T_4$};
\node at (5.5,-0.6) {$T_5$};
\foreach \from/\to in {1/2,2/3,1/3,2/4,2/5,2/6,4/7,6/8,6/9,5/10,5/11,5/12,3/13,14/13,1/15,15/16,16/17,17/18,19/20,20/21,21/22,23/24,26/27,26/28,29/30,34/31,34/32,34/33} {
 \draw (\from) -- (\to);}
 \foreach \source / \dest in {}
   \path[matched edge] (\source) -- (\dest);
\end{tikzpicture}
\caption{Unicyclic graph of Type $II$ and the support of its subtrees.}\label{f172}
\end{center}
\end{scriptsize}
 \end{center}
\end{figure}
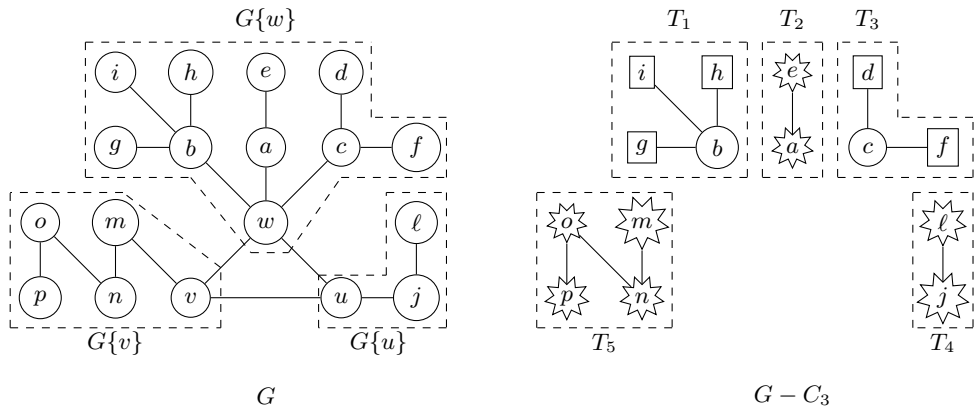

As an example, consider $G$ the unicyclic graph of Figure \ref{f172}. We
first notice that $G$ is a unicyclic graph of Type $II$, because
$v\in{Supp(G\lbrace{v}\rbrace)}=\lbrace{n,p,v}\rbrace$,
$w\in{Supp(G\lbrace{w}\rbrace)}=\lbrace{e,g,h,i,w,d,f}\rbrace$ and
$u\in{Supp(G\lbrace{u}\rbrace)}=\lbrace{u,\ell}\rbrace$. Then by Corollary
\ref{c3} we obtain $G$ is a unicyclic graph of Type $II$. Moreover, we have
$G-C_3=\bigcup\limits_{i=1}^{5}T_i$, where
$T_1=G{\langle\lbrace{b,g,h,i}\rbrace\rangle}$,
$T_2=G{\langle\lbrace{a,e}\rbrace\rangle}$,
$T_3=G{\langle\lbrace{c,f,d}\rbrace\rangle}$,
$T_4=G{\langle\lbrace{j,\ell}\rbrace\rangle}$ and
$T_5=G{\langle\lbrace{m,n,o,p}\rbrace\rangle}$ (see Figure \ref{f172}).
\\ Since $T_2$, $T_4$ and $T_5$ have perfect matching, so they are non-singular
and have empty support.  $T_1$ and $T_3$ do not have perfect matching, so
they are singular and computing their supports we obtain the Table
\ref{tab1}.

\begin{center}
\begin{table}[h]
\begin{tabular}{|l|l|l|}
  \hline
   $G-C_3$  & Support & $N$-vertices\\
  \hline
  $T_1$ & $Supp(T_1)=\lbrace{g,h,i}\rbrace$ &  $V(\mathcal{F}_N(T_1))=\emptyset$\\
  \hline
  $T_2$ & $Supp(T_2)=\emptyset$ &  $V(\mathcal{F}_N(T_2))=\lbrace{a,e}\rbrace$\\
  \hline
  $T_3$ & $Supp(T_3)=\lbrace{d,f}\rbrace$  &$V(\mathcal{F}_N(T_3))=\emptyset$ \\
  \hline
  $T_4$ &  $Supp(T_4)=\emptyset$  &$V(\mathcal{F}_N(T_4))=\lbrace{j,\ell}\rbrace$\\
  \hline
  $T_5$ &  $Supp(T_5)=\emptyset$  &$V(\mathcal{F}_N(T_5))=\lbrace{m,n,o,p}\rbrace$ \\
  \hline
\end{tabular}
\caption{Support and $N$-vertices of the subtrees $T_1$, $T_2$, $T_3$, $T_4$ and $T_5$.}\label{tab1}
\end{table}
\end{center}

\noindent Therefore, by Theorem \ref{propalpha2} we have that the independence
number of $G$ is given by:
\begin{eqnarray*}
\alpha{(G)}&=&\floor*{\frac{\vert{V(C)}\vert}{2}}+\sum\limits^{k}_{i=1}\vert{Supp(T_i)}\vert{+}\frac{\vert{V(\mathcal{F}_N(T_i))}\vert}{2}\\
&=&\floor*{\frac{3}{2}}+3+\frac{2}{2}+2+\frac{4}{2}+\frac{2}{2}=10.
\end{eqnarray*}
\noindent We observe that $I=\lbrace{g,h,i,e,d,f,p,n,\ell,v}\rbrace$ is a
maximum independent set of $G$ and $\vert{I}\vert=10$.

\section{Matching number of Unicyclic graphs}\label{sec:nu}
In this section we obtain closed formulas for the matching
number of unicyclic graphs.
\begin{defn}
Let $M$ be a matching in the graph $G$. An $M$-alternating path is a path
that alternates edges in $M$ and edges that are not in $M$. An $M$-augmenting
path is an $M$-alternating path, if it begins and ends at vertices non
saturated by $M$.
\end{defn}

Consider the matching
$M=\lbrace{\lbrace{d,c}\rbrace,\lbrace{a,b}\rbrace,\lbrace{v,w}\rbrace}\rbrace$
in graph $G$ of Figure \ref{fm1}. As the path $P_1=(e,d,c,b,a)$ is an
$M$-alternating path because its edges alternate outside and within the
matching $M$. Now the path $P_2=(u,v,w,o)$ is $M$-augmenting, because it is
an $M$-alternating path and starts and ends at vertices non saturated by $M$
(vertices u and o).

\usetikzlibrary{shapes,snakes}
\tikzstyle{vertex}=[circle,draw,minimum size=5pt,inner sep=10pt]
\tikzstyle{edge} = [draw,thick,-]
\tikzstyle{matched edge} = [draw,snake=zigzag,line width=1pt,-]
\begin{figure}[h!]
\begin{center}
\begin{scriptsize}
\begin{center}
\begin{tikzpicture}[scale=1,auto,swap]
\node[draw,circle,label=below left:] (1) at (0,0) {$u$};
\node[draw,circle,label=below left:] (2) at (0,1) {$v$};
\node[draw,circle,label=below left:] (3) at (1,1) {$w$};
\node[draw,circle,label=below left:] (4) at (1,0) {$x$};
\node[draw,circle,label=below left:] (5) at (-1,2) {$a$};
\node[draw,circle,label=below left:] (6) at (-1,1) {$b$};
\node[draw,circle,label=below left:] (7) at (-2,2) {$c$};
\node[draw,circle,label=below left:] (8) at (-2,1) {$d$};
\node[draw,circle,label=below left:] (18) at (2,1) {$o$};
\node[draw,circle,label=below left:] (19) at (1,2) {$k$};
\node[draw,circle,label=below left:] (20) at (2,2) {$p$};
\node[draw,circle,label=below left:] (21) at (3,2) {$q$};
\node[draw,circle,label=below left:] (22) at (-3,1) {$e$};
\draw[dashed](-0.3,1.3)--(2.3,1.3);
\draw[dashed](-0.3,1.3)--(-0.3,-0.3);
\draw[dashed](-0.3,-0.3)--(0.33,-0.3);
\draw[dashed](0.33,-0.3)--(0.33,0.7);
\draw[dashed](0.33,0.7)--(2.3,0.7);
\draw[dashed](2.3,0.7)--(2.3,1.3);
\draw[dashed](-3.3,1.3)--(-2.3,1.3);
\draw[dashed](-2.3,1.3)--(-2.3,2.3);
\draw[dashed](-2.3,2.3)--(-1.8,2.3);
\draw[dashed](-1.8,2.3)--(-1.3,1.7);
\draw[dashed](-1.3,1.7)--(-1.3,2.3);
\draw[dashed](-1.3,2.3)--(-0.7,2.3);
\draw[dashed](-0.7,2.3)--(-0.7,0.7);
\draw[dashed](-0.7,0.7)--(-1.3,0.7);
\draw[dashed](-1.3,0.7)--(-1.7,1.3);
\draw[dashed](-1.7,1.3)--(-1.7,0.7);
\draw[dashed](-1.7,0.7)--(-3.3,0.7);
\draw[dashed](-3.3,0.7)--(-3.3,1.3);
\node at (0.5,-1) {$G$};
\node at (0,-0.5) {$P_2$};
\node at (-3.5,1) {$P_1$};
\foreach \from/\to in {1/2,2/3,3/4,4/1,2/5,5/6,6/7,7/8,18/19,18/20,18/21,3/18,8/22,7/8,5/6,2/3} {
 \draw (\from) -- (\to);}
 \foreach \source / \dest in {7/8,5/6,2/3}
   \path[matched edge] (\source) -- (\dest);
\end{tikzpicture}
\end{center}
\end{scriptsize}
 \end{center}
\caption{$M$-alternating path and $M$-augmenting path.}\label{fm1}
\end{figure}
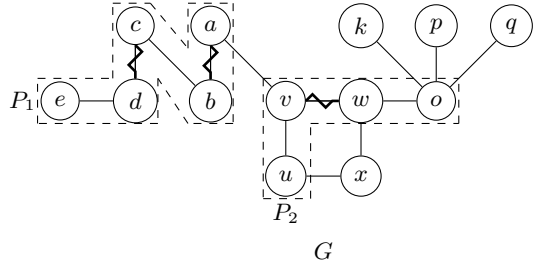
The following is a classic result, it characterizes maximum matchings in a
graph $G$.
\begin{lemat}(Berge, 1957)\cite{berge}\label{tmax}
A matching $M$ is maximum in $G$ if and only if $G$ does not have an
$M$-augmenting path.
\end{lemat}

We now give a closed formula for the matching number of unicyclic graphs $G$
of Type $I$. This formula depends on the core and $N$-vertices of subtrees.

\begin{teore}\label{propnu1} If $G$ is a unicyclic
graph of Type $I$ and $G\lbrace{v}\rbrace$ its pendant tree such that
$v\notin{Supp(G\lbrace{v}\rbrace)}$ then
$$\nu(G)=\vert{Core(G\lbrace{v}\rbrace)}\vert{+}\vert{Core(G-G\lbrace{v}\rbrace)}\vert{+}\frac{\vert{V(\mathcal{F}_N(G\lbrace{v}\rbrace))}\vert+\vert{V(\mathcal{F}_N(G-G\lbrace{v}\rbrace))}\vert}{2}.$$
\end{teore}

\begin{proof}
Let $M_{1}\in\mathcal{M}(G{\lbrace{v}\rbrace})$ and
$M_2\in{\mathcal{M}(G-G{\lbrace{v}\rbrace})}$. Let
$u,w\in{N(v)\cap{V(G-G\lbrace{v}\rbrace)}}$ and $M=M_1\cup{M_2}$. We will
prove that $M\in\mathcal{M}(G)$. Suppose that $M\notin\mathcal{M}(G)$. Hence,
by Lemma \ref{tmax}, there is an $M$-augmenting path, denoted by $P$. Notice
that $P$ is neither totally contained in $G{\lbrace{v}\rbrace}$ nor totally
contained in  $G-G\lbrace{v}\rbrace$, otherwise $P$ would be $M_1$-augmenting
or $M_2$-augmenting, which is a contradiction, because $M_{1}\in
\mathcal{M}(G{\lbrace{v}\rbrace})$ and $M_2\in
{\mathcal{M}(G-G{\lbrace{v}\rbrace})}$, respectively. Moreover $P$ does not
contain the edges $\lbrace{u,v}\rbrace$ and the edge $\lbrace{w,v}\rbrace$
simultaneous, because $\lbrace{u,v}\rbrace,\lbrace{w,v}\rbrace\notin{M}$.
Then $P$ starts at a vertex $a_1$ of $G{\lbrace{v}\rbrace}$  and ends at a
vertex $b_{2s}$ of $G-G{\lbrace{v}\rbrace}$.  Now we notice that the path $P$
contains the edge $\lbrace{u,v}\rbrace$ or the edge $\lbrace{w,v}\rbrace$.
Suppose that $P$ contains the edge $\lbrace{u,v}\rbrace$, then we have
$P=(a_1,a_2,\ldots,a_{2k},v,u,b_1,b_2,\ldots,b_{2s})$(see Figure \ref{fm5}).

\usetikzlibrary{shapes,snakes}
\tikzstyle{vertex}=[circle,draw,minimum size=5pt,inner sep=10pt]
\tikzstyle{edge} = [draw,thick,-]
\tikzstyle{matched edge} = [draw,snake=zigzag,line width=1pt,-]
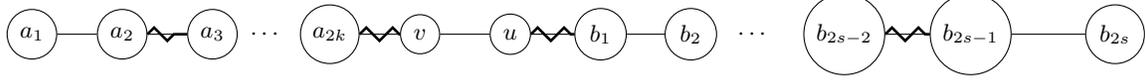
\begin{figure}[h!]
\begin{scriptsize}
\begin{center}
\begin{tikzpicture}[scale=1.2,auto,swap]
\node[draw,circle,label=below left:] (1) at (-2,0) {$a_{1}$};
\node[draw,circle,label=below left:] (2) at (-1,0) {$a_{2}$};
\node[draw,circle,label=below left:] (3) at (0,0) {$a_{3}$};
\node[draw,circle,label=below left:] (4) at (1.3,0) {$a_{2k}$};
\node[draw,circle,label=below left:] (5) at (2.3,0) {$v$};
\node[draw,circle,label=below left:] (6) at (3.3,0) {$u$};
\node[draw,circle,label=below left:] (7) at (4.3,0) {$b_{1}$};
\node[draw,circle,label=below left:] (8) at (5.3,0) {$b_{2}$};
\node[draw,circle,label=below left:] (9) at (7,0) {$b_{2s-2}$};
\node[draw,circle,label=below left:] (10) at (8.4,0) {$b_{2s-1}$};
\node[draw,circle,label=below left:] (11) at (10,0) {$b_{2s}$};
\node at (0.6,0) {$\cdots$};
\node at (6,0) {$\cdots$};
\foreach \from/\to in {1/2,5/6,7/8,10/11,2/3,4/5,6/7,9/10}{
 \draw (\from) -- (\to);}
 \foreach \source / \dest in {2/3,4/5,6/7,9/10}
   \path[matched edge] (\source) -- (\dest);
\end{tikzpicture}
\end{center}
\end{scriptsize}
\caption{$M$-augmenting path $P$.}\label{fm5}
\end{figure}

\noindent Let
$B_1=\lbrace{\lbrace{a_{2k},v}\rbrace,\lbrace{u,b_1}\rbrace}\rbrace\cup\bigcup\limits_{i=1}^{k-1}\lbrace\lbrace{a_{2i},a_{2i+1}}\rbrace\rbrace\cup\bigcup\limits_{j=1}^{s-1}\lbrace\lbrace{b_{2j},b_{2j+1}}\rbrace\rbrace$
and
$B_2=\lbrace{\lbrace{u,v}\rbrace}\rbrace\cup\bigcup\limits_{i=1}^{k}\lbrace\lbrace{a_{2i-1},a_{2i}}\rbrace\rbrace\cup\bigcup\limits_{j=1}^{s}\lbrace\lbrace{b_{2j-1},b_{2j}}\rbrace\rbrace$.
We observe that $B_1\subseteq{M}$ and $B_2\cap{M}=\emptyset$. Let $M'$ be a
matching in $G$ given by $M'=(M\cup{B_2})\setminus{B_1}$. As
$M^{'}\cap{E(G\lbrace{v}\rbrace)}$ is a matching in $G\lbrace{v}\rbrace$, we
see that it does not saturate $v$, because
$v\notin{Supp(G\lbrace{v}\rbrace)}. $ Hence $M^{'}\cap{E(G\lbrace{v}\rbrace)}
\notin\mathcal{M}(G\lbrace{v}\rbrace)$ (see Lemma \ref{t11}). That is
$\vert{M^{'}\cap{E(G\lbrace{v}\rbrace)}\vert}<{\vert{M_1}\vert}$. We have
\begin{small}
\begin{eqnarray}\label{equa2}
\vert{M^{'}\cap{E(G\lbrace{v}\rbrace)}}\vert&=&\vert{((M\cup{B_2})\setminus{B_1})\cap{E(G\lbrace{v}\rbrace)}}\vert\nonumber\\
&=&\vert{((M_1\cup{M_2}\cup{B_2})\setminus{B_1})\cap{E(G\lbrace{v}\rbrace)}}\vert\nonumber\\
&=&\vert{((M_1\setminus{B_1})\cup{(M_2\setminus{B_1})}\cup{(B_2\setminus{B_1})})\cap{E(G\lbrace{v}\rbrace)}}\vert\nonumber\\
&=&\vert{((M_1\setminus{B_1})\cup{(M_2\setminus{B_1})}\cup{B_2})\cap{E(G\lbrace{v}\rbrace)}}\vert\nonumber\\
&=&\vert{((M_1\setminus{B_1})\cap{E(G\lbrace{v}\rbrace)})\cup({(M_2\setminus{B_1})}\cap{E(G\lbrace{v}\rbrace)})\cup({B_2}\cap{E(G\lbrace{v}\rbrace)})}\vert\nonumber\\
&=&\vert{((M_1\cap{E(G\lbrace{v}\rbrace)})\setminus{B_1})\cup({(M_2\cap{E(G\lbrace{v}\rbrace)})}\setminus{B_1})\cup({B_2}\cap{E(G\lbrace{v}\rbrace)})}\vert\nonumber\\
&=&\vert{(M_1\setminus{B_1})\cup({B_2}\cap{E(G\lbrace{v}\rbrace)})}\vert\nonumber\\
&=&\vert{(M_1\setminus{B_1})}\vert+\vert{({B_2}\cap{E(G\lbrace{v}\rbrace)})}\vert\nonumber\\
&=&\left\vert{M_1}\right\vert-\left\vert{M_1\cap{B_1}}\right\vert+\vert{({B_2}\cap{E(G\lbrace{v}\rbrace)})}\vert\nonumber\\
&=&\left\vert{M_1}\right\vert-\left\vert{\lbrace\lbrace{a_{2k},v}\rbrace\rbrace\cup\bigcup\limits_{i=1}^{k-1}\lbrace\lbrace{a_{2i},a_{2i+1}}\rbrace\rbrace}\right\vert+\left\vert{\bigcup\limits_{i=1}^{k}\lbrace\lbrace{a_{2i-1},a_{2i}}\rbrace\rbrace}\right\vert\nonumber\\
&=&\left\vert{M_1}\right\vert-k+k=\vert{M_1}\vert.
\end{eqnarray}
\end{small}
By equation \eqref{equa2} we have $\vert{M_1}\vert>\vert{M^{'} \cap
{E(G\lbrace{v}\rbrace)}}\vert=\vert{M_1}\vert$, which is a contradiction.

The case where $P$ contains the edge $\lbrace{w,v}\rbrace$ is  analogous.
Therefore $M\in\mathcal{M}(G)$.  Using Lemma \ref{t35} we have
\begin{eqnarray*}
\nu{(G\lbrace{v}\rbrace)}&=&\vert{Core(G\lbrace{v}\rbrace)}\vert{+}\frac{\vert{V(\mathcal{F}_N(G\lbrace{v}\rbrace))}\vert}{2}\mbox{ and}\\
\nu{(G-G\lbrace{v}\rbrace)}&=&\vert{Core(G-G\lbrace{v}\rbrace)}\vert{+}\frac{\vert{V(\mathcal{F}_N(G-G\lbrace{v}\rbrace))}\vert}{2}.
\end{eqnarray*}

Therefore, the matching number of $G$ is given by

\begin{eqnarray*}
\nu(G)&=&\vert{M_1}\vert+{\vert{M_2}\vert}
=\nu{(G\lbrace{v}\rbrace)}+\nu{(G-G\lbrace{v}\rbrace)}\\
&=&\vert{Core(G\lbrace{v}\rbrace)}\vert{+}\vert{Core(G-G\lbrace{v}\rbrace)}\vert{+}\frac{\vert{V(\mathcal{F}_N(G\lbrace{v}\rbrace))}\vert+\vert{V(\mathcal{F}_N(G-G\lbrace{v}\rbrace))}\vert}{2}.
\end{eqnarray*}
\end{proof}

As an example of Theorem \ref{propnu1}, consider  the unicyclic graph $G$ of
Figure \ref{figurmatch4}. To see that it is of Type $I$, we notice that
$v\notin{Supp(G\lbrace{v}\rbrace)=\lbrace{a,b,j}\rbrace}$, then by
Proposition \ref{t4}, we have $G$ is a unicyclic graph of Type $I$. Moreover,
$Supp(G-G\lbrace{v}\rbrace)=\lbrace{e,f,i}\rbrace$,
$Core(G\lbrace{v}\rbrace)=\lbrace{c}\rbrace$,
$V(\mathcal{F}_N(G\lbrace{v}\rbrace))=\lbrace{v,d,\ell,m}\rbrace$,
$Core(G-G\lbrace{v}\rbrace)=\lbrace{g,h}\rbrace$ and
$V(\mathcal{F}_N(G-G\lbrace{v}\rbrace))=\lbrace{o,n}\rbrace$.

\usetikzlibrary{shapes,snakes}
\tikzstyle{vertex}=[circle,draw,minimum size=5pt,inner sep=10pt]
\tikzstyle{edge} = [draw,thick,-]
\tikzstyle{matched edge} = [draw,snake=zigzag,line width=1pt,-]
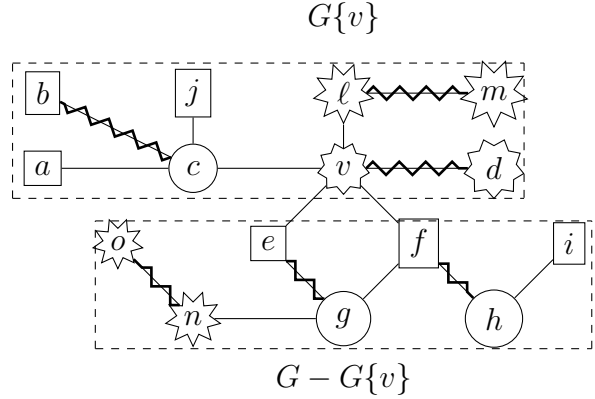
\begin{figure}[h!]
\begin{center}
\begin{tikzpicture}[scale=1,auto,swap]
\node[draw,circle,label=below left:] (1) at (-2,-1) {$c$};
\node[draw,rectangle,label=below left:] (2) at (-4,-1) {$a$};
\node[draw,rectangle,label=below left:] (3) at (-4,0) {$b$};
\node[draw,star,star points=9,star point ratio=0.75,label=below left:] (4) at (0,-1) {$v$};
\node[draw,star,star points=9,star point ratio=0.75,label=below left:] (5) at (2,-1) {$d$};
\node[draw,rectangle,label=below left:] (6) at (-1,-2) {$e$};
\node[draw,rectangle,label=below left:] (7) at (1,-2) {$f$};
\node[draw,circle,label=below left:] (8) at (0,-3) {$g$};
\node[draw,circle,label=below left:] (9) at (2,-3) {$h$};
\node[draw,rectangle,label=below left:] (10) at (3,-2) {$i$};
\node[draw,rectangle,label=below left:] (11) at (-2,0) {$j$};
\node[draw,star,star points=9,star point ratio=0.6,label=below left:] (12) at (0,0) {$\ell$};
\node[draw,star,star points=9,star point ratio=0.6,label=below left:] (13) at (2,0) {$m$};
\node[draw,star,star points=9,star point ratio=0.6,label=below left:] (14) at (-2,-3) {$n$};
\node[draw,star,star points=9,star point ratio=0.6,label=below left:] (15) at (-3,-2) {$o$};
\draw[dashed] (-3.3,-1.7) -- (3.3,-1.7);
\draw[dashed] (-3.3,-3.4) -- (3.3,-3.4);
\draw[dashed] (-3.3,-3.4) -- (-3.3,-1.7);
\draw[dashed] (3.3,-3.4) -- (3.3,-1.7);
\draw[dashed] (-4.4,-1.4) -- (2.4,-1.4);
\draw[dashed] (2.4,0.4) -- (2.4,-1.4);
\draw[dashed] (-4.4,0.4)-- (2.4,0.4);
\draw[dashed] (-4.4,-1.4)--(-4.4,0.4);
\node at (0,-3.8) {$G-G\lbrace{v}\rbrace$};
\node at (0,1) {$G\lbrace{v}\rbrace$};
\foreach \from/\to in {1/2,1/3,1/4,4/5,4/6,4/7,8/6,8/7,9/7,9/10,11/1,4/12,12/13,8/14,14/15} {
\draw (\from) -- (\to);}
\foreach \source / \dest in {1/3,4/5,12/13,14/15,6/8,7/9}
\path[matched edge] (\source) -- (\dest);
\end{tikzpicture}
\caption{Unicyclic graph of Type $I$ and support of subtrees.}\label{figurmatch4}
\end{center}
\end{figure}

\noindent Therefore, by Theorem \ref{propnu1}, we have that the matching
number of $G$ is given by
\begin{eqnarray*}
\nu(G)&=&\vert{Core(G\lbrace{v}\rbrace)}\vert{+}\vert{Core(G-G\lbrace{v}\rbrace)}\vert{+}\frac{\vert{V(\mathcal{F}_N(G\lbrace{v}\rbrace))}\vert+\vert{V(\mathcal{F}_N(G-G\lbrace{v}\rbrace))}\vert}{2}\\
&=&1+\frac{4}{2}+2+\frac{2}{2}=6.
\end{eqnarray*}

\noindent We  point out  that
$M=\lbrace{\lbrace{b,c}\rbrace,\lbrace{v,d}\rbrace,\lbrace{\ell,m}\rbrace,
\lbrace{n,o}\rbrace,\lbrace{e,g}\rbrace,\lbrace{f,h}\rbrace}\rbrace$ is a
maximum matching of $G$ and $\vert{M}\vert=6$.

We now present a similar result for the matching number of unicyclic graphs
$G$ of Type $II$.

\begin{teore}\label{propnu2}
Let $G$ be a unicyclic graph and $C$ its cycle. Let
$G-C=\bigcup\limits^{k}_{i=1}{T_i}$, where $T_i$ is a connected component of
$G-C$. If $G$ is a unicyclic graph of Type $II$ then
\begin{equation*}
\nu(G)=\floor*{\frac{\vert{V(C)}\vert}{2}}+\sum\limits^{k}_{i=1}\vert{Core(T_i)}\vert+\frac{\vert{V(\mathcal{F}_N(T_i))}\vert}{2}.
\end{equation*}
\end{teore}
\begin{proof}
For each pendant tree $G\lbrace{v}\rbrace$, with $G\lbrace{v}\rbrace\neq{v}$,
choose an $M_v\in\mathcal{M}(G\lbrace{v}\rbrace)$ that does not saturate $v$. 
Note that this maximum matching exists because $G$ is a unicyclic graph of Type
$II$, and this implies that $v\in{Supp(G\lbrace{v}\rbrace)}$ for all
$v\in{V(C)}$. Hence, by Lemma \ref{t11}, we have that $M_v$ exists. Choose an
$M_c\in\mathcal{M}(C)$ and let $M=M_c\cup\left(\bigcup\limits_{v}M_v\right)$.

We will show that $M\in\mathcal{M}(G)$.
Suppose by contradiction that $M\notin\mathcal{M}(G)$. Then, by Lemma \ref{tmax} there exists a $M$-augmenting path denoted by $P$ in $G$. 
Note that if $V(P)\subseteq V(C)$ then $P$ would be a $M_c$-augmenting path in $C$ which is a contradiction 
because $M_c\in\mathcal{M}(C)$. 
Now, if $V(P)\subseteq V(G\{v\})$ then $P$ would be a $M_v$-augmenting path  in $G\lbrace{v}\rbrace$,
which is a contradiction because $M_v\in\mathcal{M}(G\lbrace{v}\rbrace)$.
 
Note that the only way to obtain an $M$-augmenting path $P$ in $G$ is if we
start the path $P$ at a vertex $u\in{V(G\lbrace{v}\rbrace)}$, with
$u\neq{v}$, and end at a vertex $z\notin V(G\{v\})$. If that happens, there would
be an $M_v$-alternating path starting in $u$ and ending in $v$ contained in $P$. But
since $M_v$ does not saturate $v$, actually we would obtain a $M_v$-augmenting path in $G\lbrace{v}\rbrace$,  which is a
contradiction, because $M_v\in\mathcal{M}(G\lbrace{v}\rbrace)$. Therefore,
$M\in\mathcal{M}(G)$.

\noindent By Lemma \ref{t35}, we have
$$\nu(T_i)=\vert{Core(T_i)}\vert+\frac{\vert{V(\mathcal{F}_N(T_i))}\vert}{2}.$$
Moreover, we have  $\nu(C)=\floor*{\frac{\vert{V(C)}\vert}{2}}$. Since $v$ is
not saturated by $M_v$ in $G\lbrace{v}\rbrace$, we see that
$M_v\in\mathcal{M}(G\lbrace{v}\rbrace-v)$, that is,
$\nu(G\lbrace{v}\rbrace)=\nu(G\lbrace{v}\rbrace-v)$. Therefore, we have that
the matching number is given by:
\begin{eqnarray*}
\nu(G)&=&\vert{M_c}\vert+\sum\limits_{G\lbrace{v}\rbrace\neq{v}}\vert{M_v}\vert
      =\nu(C)+\sum\limits_{G\lbrace{v}\rbrace\neq{v}}\nu(G\lbrace{v}\rbrace)\\
      &=&\nu(C)+\sum\limits_{G\lbrace{v}\rbrace\neq{v}}\nu(G\lbrace{v}\rbrace-v)
      =\nu(C)+\sum\limits^{k}_{i=1}\nu(T_i)\\
      &=&\floor*{\frac{\vert{V(C)}\vert}{2}}+\sum\limits^{k}_{i=1}\vert{Core(T_i)}\vert+\frac{\vert{V(\mathcal{F}_N(T_i))}\vert}{2}.
\end{eqnarray*}

\end{proof}

Consider $G$ the unicyclic graph of Figure \ref{fm3}. We see that $G$ is a
unicyclic graph of Type $II$, because
$a\in{Supp(G\lbrace{a}\rbrace)=\lbrace{a,j,l,m}\rbrace}$,
$b\in{Supp(G\lbrace{b}\rbrace)=\lbrace{b,h,g}\rbrace}$,
$c\in{Supp(G\lbrace{c}\rbrace)=\lbrace{c,o}\rbrace}$,
$d\in{Supp(G\lbrace{d}\rbrace)=\lbrace{d,u,v,w,t}\rbrace}$ and
$e\in{Supp(G\lbrace{e}\rbrace)=\lbrace{e}\rbrace}$. By Corollary \ref{c3}, we
have that $G$ is of Type $II$. Notice that
$G-C_5=\bigcup\limits_{i=1}^{4}T_i$, where
$T_1=G{\langle\lbrace{f,g,h}\rbrace\rangle}$,
$T_2=G{\langle\lbrace{n,o,p,q}\rbrace\rangle}$,
$T_3=G{\langle\lbrace{r,s,t,u,v,w}\rbrace\rangle}$ and
$T_4=G{\langle\lbrace{j,i,l,m}\rbrace\rangle}$  (see Figure \ref{fm3}). We
see that $T_2$ has perfect matching, then $T_2$ is non singular and so $T_2$
has empty support. Moreover, we have that $T_1$, $T_3$ and $T_4$ do not have
perfect matchings, then they are singular and their supports is given in
Table \ref{tab2}.

\begin{center}
\begin{table}[h]
\begin{tabular}{|l|l|l|}
  \hline
    Support & Core & $N$-vertices\\
  \hline
  $Supp(T_1)=\lbrace{h,g}\rbrace$     & $Core(T_1)=\lbrace{f}\rbrace$ & $V(\mathcal{F}_N(T_1))=\emptyset$ \\
  \hline
  $Supp(T_2)=\emptyset$       & $Core(T_2)=\emptyset$ & $V(\mathcal{F}_N(T_2))=\lbrace{n,o,p,q}\rbrace$\\
  \hline
  $Supp(T_3)=\lbrace{t,v,u,w}\rbrace$         & $Core(T_3)=\lbrace{r,s}\rbrace$ &  $V(\mathcal{F}_N(T_3)))=\emptyset$\\
  \hline
    $Supp(T_4)=\lbrace{j,l,m}\rbrace$ & $Core(T_4))=\lbrace{i}\rbrace$ &  $V(\mathcal{F}_N(T_4))=\emptyset$\\
  \hline
\end{tabular}
\caption{Support, core and $N$-vertices of the trees $T_1$, $T_2$, $T_3$ and $T_4$.}\label{tab2}
\end{table}
\end{center}
\usetikzlibrary{shapes,snakes}
\tikzstyle{vertex}=[circle,draw,minimum size=1pt,inner sep=1pt]
\tikzstyle{edge} = [draw,thick,-]
\tikzstyle{matched edge} = [draw,snake=zigzag,line width=1pt,-]
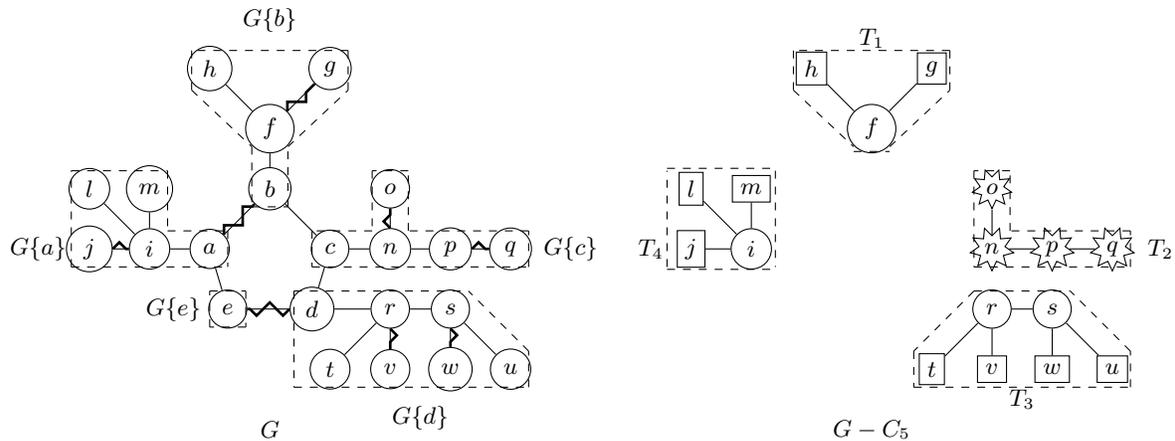
\begin{figure}[h!]
\begin{center}
\begin{scriptsize}
\begin{center}
\begin{tikzpicture}[scale=0.8,auto,swap]
\node[draw,circle,label=below left:] (1) at (0,0) {$a$};
\node[draw,circle,label=below left:] (2) at (1,1) {$b$};
\node[draw,circle,label=below left:] (3) at (2,0) {$c$};
\node[draw,circle,label=below left:] (4) at (1.7,-1) {$d$};
\node[draw,circle,label=below left:] (5) at (0.3,-1) {$e$};
\node[draw,circle,label=below left:] (6) at (1,2) {$f$};
\node[draw,circle,label=below left:] (7) at (2,3) {$g$};
\node[draw,circle,label=below left:] (8) at (0,3) {$h$};
\node[draw,circle,label=below left:] (9) at (-1,0) {$i$};
\node[draw,circle,label=below left:] (10) at (-2,0) {$j$};
\node[draw,circle,label=below left:] (11) at (-2,1) {$l$};
\node[draw,circle,label=below left:] (12) at (-1,1) {$m$};
\node[draw,circle,label=below left:] (13) at (3,0) {$n$};
\node[draw,circle,label=below left:] (14) at (3,1) {$o$};
\node[draw,circle,label=below left:] (15) at (4,0) {$p$};
\node[draw,circle,label=below left:] (16) at (5,0) {$q$};
\node[draw,circle,label=below left:] (17) at (3,-1) {$r$};
\node[draw,circle,label=below left:] (18) at (4,-1) {$s$};
\node[draw,circle,label=below left:] (19) at (2,-2) {$t$};
\node[draw,circle,label=below left:] (20) at (5,-2) {$u$};
\node[draw,circle,label=below left:] (21) at (3,-2) {$v$};
\node[draw,circle,label=below left:] (22) at (4,-2) {$w$};
\node[draw,circle,label=below left:] (23) at (11,2) {$f$};
\node[draw,rectangle,label=below left:] (24) at (12,3) {$g$};
\node[draw,rectangle,label=below left:] (25) at (10,3) {$h$};
\node[draw,circle,label=below left:] (26) at (9,0) {$i$};
\node[draw,rectangle,label=below left:] (27) at (8,0) {$j$};
\node[draw,rectangle,label=below left:] (28) at (8,1) {$l$};
\node[draw,rectangle,label=below left:] (29) at (9,1) {$m$};
\node[draw,star,star points=9,star point ratio=0.6,label=below left:] (30) at (13,0) {$n$};
\node[draw,star,star points=9,star point ratio=0.6,label=below left:] (31) at (13,1) {$o$};
\node[draw,star,star points=9,star point ratio=0.6,label=below left:] (32) at (14,0) {$p$};
\node[draw,star,star points=9,star point ratio=0.6,label=below left:] (33) at (15,0) {$q$};
\node[draw,circle,label=below left:] (34) at (13,-1) {$r$};
\node[draw,circle,label=below left:] (35) at (14,-1) {$s$};
\node[draw,rectangle,label=below left:] (36) at (12,-2) {$t$};
\node[draw,rectangle,label=below left:] (37) at (15,-2) {$u$};
\node[draw,rectangle,label=below left:] (38) at (13,-2) {$v$};
\node[draw,rectangle,label=below left:] (39) at (14,-2) {$w$};
\draw[dashed] (1.7,-0.3)--(5.3,-0.3);
\draw[dashed] (5.3,-0.3)--(5.3,0.3);
\draw[dashed] (5.3,0.3)--(3.3,0.3);
\draw[dashed] (3.3,0.3)--(3.3,1.3);
\draw[dashed] (3.3,1.3)--(2.7,1.3);
\draw[dashed] (2.7,1.3)--(2.7,0.3);
\draw[dashed] (2.7,0.3)--(1.7,0.3);
\draw[dashed] (1.7,0.3)--(1.7,-0.3);
\draw[dashed] (-2.3,-0.3)--(0.3,-0.3);
\draw[dashed] (-2.3,-0.3)--(-2.3,1.3);
\draw[dashed] (-2.3,1.3)--(-0.7,1.3);
\draw[dashed] (-0.7,1.3)--(-0.7,0.3);
\draw[dashed] (-0.7,0.3)--(0.3,0.3);
\draw[dashed] (0.3,0.3)--(0.3,-0.3);
\draw[dashed] (1.4,-0.7)--(4.3,-0.7);
\draw[dashed] (1.4,-0.7)--(1.4,-2.3);
\draw[dashed] (1.4,-2.3)--(5.3,-2.3);
\draw[dashed] (5.3,-2.3)--(5.3,-1.7);
\draw[dashed] (5.3,-1.7)--(4.3,-0.7);
\draw[dashed] (0.7,0.7)--(1.3,0.7);
\draw[dashed] (1.3,0.7)--(1.3,1.7);
\draw[dashed] (1.3,1.7)--(2.3,2.64);
\draw[dashed] (2.3,2.64)--(2.3,3.3);
\draw[dashed] (2.3,3.3)--(-0.3,3.3);
\draw[dashed] (-0.3,3.3)--(-0.3,2.64);
\draw[dashed] (-0.3,2.64)--(0.7,1.7);
\draw[dashed] (0.7,1.7)--(0.7,0.7);
\draw[dashed] (0,-0.7)--(0.6,-0.7);
\draw[dashed] (0.6,-0.7)--(0.6,-1.3);
\draw[dashed] (0.6,-1.3)--(0,-1.3);
\draw[dashed] (0,-1.3)--(0,-0.7);
\node at (1,-3) {$G$};
\node at (11,-3) {$G-C_5$};
\node at (1,3.8) {$G\lbrace{b}\rbrace$};
\node at (6,0) {$G\lbrace{c}\rbrace$};
\node at (-2.85,0) {$G\lbrace{a}\rbrace$};
\node at (-0.6,-1) {$G\lbrace{e}\rbrace$};
\node at (3.5,-2.8) {$G\lbrace{d}\rbrace$};
\draw[dashed] (12.7,-0.3)--(15.3,-0.3);
\draw[dashed] (15.3,-0.3)--(15.3,0.3);
\draw[dashed] (15.3,0.3)--(13.3,0.3);
\draw[dashed] (13.3,0.3)--(13.3,1.3);
\draw[dashed] (13.3,1.3)--(12.7,1.3);
\draw[dashed] (12.7,1.3)--(12.7,-0.3);
\draw[dashed] (7.6,-0.34)--(9.4,-0.34);
\draw[dashed] (7.6,-0.3)--(7.6,1.3);
\draw[dashed] (7.6,1.34)--(9.4,1.34);
\draw[dashed] (9.4,1.3)--(9.4,-0.3);
\draw[dashed] (12.7,-0.7)--(11.7,-1.7);
\draw[dashed] (12.7,-0.7)--(14.3,-0.7);
\draw[dashed] (11.7,-1.7)--(11.7,-2.3);
\draw[dashed] (11.7,-2.3)--(15.3,-2.3);
\draw[dashed] (15.3,-2.3)--(15.3,-1.7);
\draw[dashed] (15.3,-1.7)--(14.3,-0.7);
\draw[dashed] (10.7,1.62)--(11.3,1.62);
\draw[dashed] (11.3,1.7)--(12.3,2.64);
\draw[dashed] (12.3,2.64)--(12.3,3.3);
\draw[dashed] (12.3,3.3)--(9.7,3.3);
\draw[dashed] (9.7,3.3)--(9.7,2.64);
\draw[dashed] (9.7,2.64)--(10.7,1.7);
\node at (11,3.5) {$T_1$};
\node at (15.8,0) {$T_2$};
\node at (7.33,0) {$T_4$};
\node at (13.5,-2.5) {$T_3$};
\foreach \from/\to in {1/2,2/3,3/4,4/5,5/1,6/2,23/24,23/25,35/37,35/39,30/31,30/32,32/33,26/27,26/28,26/29,6/7,6/8,9/1,9/10,9/11,9/12,13/14,13/3,13/15,15/16,4/17,17/18,17/19,17/21,18/20,18/22,34/35,34/36,34/38} {
 \draw (\from) -- (\to);}
 \foreach \source / \dest in {10/9,6/7,13/14,15/16,17/21,18/22,1/2,4/5}
   \path[matched edge] (\source) -- (\dest);
\end{tikzpicture}
\caption{Unicyclic graph of Type $II$ and its subtrees $T_1$, $T_2$, $T_3$ and $T_4$.}\label{fm3}
\end{center}
\end{scriptsize}
 \end{center}
\end{figure}

\noindent Therefore, by Theorem \ref{propnu2}, we have that the matching
number of $G$ is given by:
\begin{eqnarray*}
\nu(G)&=&\floor*{\frac{\vert{V(C)}\vert}{2}}+\sum\limits^{k}_{i=1}\vert{Core(T_i)}\vert+\frac{\vert{V(\mathcal{F}_N(T_i))}\vert}{2}.\\
&=&\floor*{\frac{5}{2}}+1+\frac{4}{2}+2+1=8.
\end{eqnarray*}

\noindent we point out that
$M=\lbrace{\lbrace{a,b}\rbrace,\lbrace{d,e}\rbrace,\lbrace{i,j}\rbrace,\lbrace{f,g}\rbrace,\lbrace{p,q}\rbrace,\lbrace{n,o}\rbrace,\lbrace{r,v}\rbrace,\lbrace{s,w}\rbrace}\rbrace$
is a maximum matching of $G$ and $\vert{M}\vert=8$.

\section*{Acknowledgments} Work supported by MATHAMSUD 18-MATH-01.
Maikon Toledo thanks CAPES for their support. V. Trevisan acknowledges
partial support of CNPq grants 409746/2016-9 and 303334/2016-9,
 and FAPERGS (Proj.\ PqG 17/2551-0001).

\end{document}